\documentclass[11pt]{amsart}
\usepackage{amsmath,amsfonts,amsthm,amssymb,amsbsy,enumerate,color} 
\usepackage[bookmarks=true]{hyperref}

\newtheorem{Theorem}{Theorem}

\newtheorem{Lemma}[Theorem]{Lemma}
\newtheorem{Remark}[Theorem]{Remark}

{\unskip\nobreak\hskip 1em plus 1fil\nobreak$\Box$
\parfillskip=0pt%
\endtrivlist}
 \def \vs { \left ( \begin {array} {c} }
    \def \ve { \end {array} \right )} 
\newtheorem{Definition}[Theorem]{Definition}

\renewcommand{\epsilon}{\varepsilon}
\DeclareMathOperator{\dvg}{div} \DeclareMathOperator{\spt}{spt}
 
\DeclareMathOperator{\graph}{graph}

    \DeclareMathOperator{\tr}{tr}
  \DeclareMathOperator{\Id}{Id}
    
    \DeclareMathOperator{\Ric}{Ric}    
    
\def\dmt{\,d\mu_t}
\def\R{\mathbb{R}}

\def\N{\mathbb{N}}

\def\d{\delta}
\def\a{\alpha}
\def\e{\epsilon}
\def\t{\tau}

\def\r{\rho}

\def\s{\sigma}

\def\l{\lambda}
\def\k{\kappa}

\def\H{\mathcal{H}}

\def\mass{\underline{\underline{M}}}
\def\var{\underline{\underline{\text{v}}}}
\def\wt{\widetilde}

\def\ov{\overline}

\def\res{\hbox{ {\vrule height .25cm}{\leaders\hrule\hskip.2cm}}\hskip5.0\mu}
\def\BV{\text{BV}}
\def\loc{\text{loc}}
\def\out{\text{out}}

\title{Null mean curvature flow and outermost MOTS}
\begin{document}
\author{Theodora Bourni \and Kristen Moore}
\thanks{Part of this work was completed while the second author was financed by the Sonderforschungsbereich \#ME3816/1-1 of the DFG}

\begin{abstract} We study the evolution of hypersurfaces in spacetime initial data sets by their null mean curvature. A theory of weak solutions is developed using the level-set approach. Starting from an arbitrary mean convex, outer untrapped hypersurface $\partial\Omega_0$, we show that there exists a weak solution to the null mean curvature flow, given as a limit of approximate solutions that are defined using the $\e$-regularization method. We show that the approximate solutions blow up on the outermost MOTS and the weak solution converges (as boundaries of finite perimeter sets) to a generalized MOTS.
\end{abstract}
\maketitle
\tableofcontents
\section{Introduction}
\noindent We consider the evolution of hypersurfaces in an initial data set $(M^{n+1},g,K)$ that arises as a spacelike hypersurface $M^{n+1}$ in a Lorentzian spacetime, $(L^{n+2},h)$, with induced metric $g$ and second fundamental form  $K$. Let $\vec{n}$ denote the future directed timelike unit normal vector field of $M\subset L$, and consider a 2-sided closed and bounded hypersurface $\Sigma^{n}\subset M^{n+1}$ with globally defined outer unit normal vector field $\nu$ in $M$. Given a smooth hypersurface immersion $F_{0}:\Sigma\rightarrow M$, the evolution of $\Sigma_{0}:=F_{0}(\Sigma)$ by null mean curvature is the one-parameter family of smooth immersions $F:\Sigma\times[0,T)\rightarrow M$ satisfying

\[\tag{$*$}
\left\lbrace
\begin{aligned}
 \frac{\partial F}{\partial t}(x,t) & = -(H+P)(x,t)\nu(x,t), \quad x\in\Sigma,\quad t\geq0, \\
                         F(x,0) & = F_{0}(x),\quad x\in\Sigma{},
\end{aligned}
\right.
\]
where $H:=\text{div}_{\Sigma_t}(\nu)$ denotes the mean curvature of $\Sigma_t:=F(\Sigma,t)$ in $M$ and $P:=\text{tr}_{\Sigma_t}K$ is the trace of $K$ over the tangent space of $\Sigma_t$. The quantity $H+P$ corresponds to the {\it null expansion} or \textit{null mean curvature} $\theta_{\Sigma_t}^+$ of $\Sigma_t$ with respect to its future directed outward null vector field $l^+:=\nu+\vec{n}$,
\begin{equation*}
\theta^+_{\Sigma_t}:=\langle\vec{H}_{\Sigma_t},l^+\rangle_h=H+P,
\end{equation*}
where $\vec{H}_{\Sigma_t}$, the mean curvature vector of $\Sigma_t$ inside the spacetime $L$, is given by 
\begin{equation*}\vec{H}_{\Sigma_t}:=H\nu-P\vec{n}.
\end{equation*} 
We will also assume that $(H+P)|_{_{\Sigma_0}}>0$ so that the hypersurface $\Sigma_{t}$ contracts under the flow. We will see below that null mean curvature flow arises as the steepest descent flow of ``area plus bulk energy $P$" with respect to the $L^2$-norm on the hypersurface. It is a generalization of mean curvature flow in that the latter corresponds to the special time-symmetric case of $(*)$, where $K\equiv0$. 

The motivation for studying this particular generalization of mean curvature flow follows from the study of black holes in general relativity. Physically, the outward null mean curvature $\theta^+_{\Sigma}$ measures the divergence of the outward directed light rays emanating from $\Sigma$. If $\theta_{\Sigma}^+$ vanishes on all of $\Sigma$, then $\Sigma$ is called a \textit{marginally outer trapped hypersurface}, or MOTS for short. MOTS play the role of apparent horizons or quasi-local black hole boundaries in general relativity, and are particularly useful for numerically modeling the dynamics and evolution of black holes. For a more detailed discussion and further references see \cite{AMS08,  AM09, AM10}.

From a mathematical point of view, MOTS are the Lorentzian analogue of minimal hypersurfaces. However, since MOTS are not stationary solutions of an elliptic variational problem, the direct method of the calculus of variations is not a viable approach to the existence theory. A successful approach to proving existence of MOTS comes from studying the blow-up set of solutions of \textit{Jang's equation}
\begin{equation}\label{Jang}
\left(g^{ij}-\dfrac{\nabla^iw\nabla^jw}{|\nabla w|^2+1}\right)\left(\dfrac{\nabla_i\nabla_jw}{\sqrt{|\nabla w|^2+1}}+K_{ij}\right)=0,
\end{equation}
for the height function $w$ of a hypersurface. This was an essential ingredient in the Schoen--Yau proof of the positive mass theorem  \cite{SY81}. In their analysis, Schoen and Yau showed that the boundary of the blow-up set  of Jang's equation consists of marginally trapped hypersurfaces. Building upon this work, existence of MOTS in compact data sets with two boundary components, such that the inner boundary is (outer) trapped and the outer boundary is (outer) untrapped, was pointed out by Schoen \cite{S04}, with proofs given by Andersson and Metzger  \cite{AM09}, and subsequently by Eichmair \cite{E09} using a different approach. 

Jang's equation also featured in the second author's study of weak solutions to the evolution by inverse null mean curvature flow  \cite{KM13}, where it was proven that the weak solution starting from any outer trapped initial hypersurface $\partial\Omega_0$ will instantly jump to a MOTS in $M\setminus\bar{\Omega}_0$. Similarly, we see below that Jang's equation plays a key role in the existence theory for weak solutions to $(*)$, as well as the ensuing application of locating MOTS in space-time initial data sets. 

The idea of using geometric evolution equations to find apparent horizons dates back to the work of Tod \cite{Tod}, who suggested using mean curvature flow to find MOTS in time symmetric slices where $K=0$ (and MOTS are minimal hypersurfaces). White \cite{W00} showed that if the initial hypersurface encloses a minimal hypersurface, the outermost such minimal hypersurface will be the stable limit of mean curvature flow. In the same paper \cite{Tod}, Tod also proposed using null mean curvature flow  in the non time-symmetric setting. Numerical results by Bernstein, Shoemaker et al. and Pasch \cite{P97} show convergence of the null mean curvature flow to a MOTS. This paper provides a mathematical justification of these numerical results

Analogous to the behavior of solutions to mean curvature flow, in general it is expected that the null mean curvature of solutions of $(*)$ will tend to infinity at some points, and that singularities will develop. This motivates our development of a theory of weak solutions to the classical flow $(*)$ in this paper, which we implement to investigate the limit of a hypersurface moving under null mean curvature flow.
To develop the weak formulation for the classical evolution $(*$), we use the level-set method and assume the evolving hypersurfaces are given by the level sets,
\begin{equation}\label{levelset1}\Sigma_{t}= \partial\{ x\in M \,\, \big| \,\, u(x)>t \}, \end{equation}
of a scalar function $u:M\rightarrow\mathbb{R}$. Then, whenever $u$ is smooth and $\nabla u\neq0$, the hypersurface flow equation $(*)$ is equivalent to the following degenerate elliptic scalar PDE
\[\tag{$**$}
 \mbox{div}_M\left( \frac{\nabla u}{|\nabla u|}\right) -\left(g^{ij}-\dfrac{\nabla^iu\nabla^ju}{|\nabla u|^2}\right)K_{ij}= \frac{-1}{|\nabla u|}.
\]

We employ the method of \textit{elliptic regularization} to solve $(**)$, and study solutions, $u_{\varepsilon}$,  of the following strictly elliptic equation 
\begin{align*}{(*_{\varepsilon})}\,\,\, \text{div}_M\left(\dfrac{\nabla u_{\varepsilon}}{\sqrt{|\nabla u_{\varepsilon}|^2+\varepsilon^2}}\right)-\left(g^{ij}-\dfrac{\nabla^iu_{\varepsilon}\nabla^ju_{\varepsilon}}{|\nabla u_{\varepsilon}|^2+\varepsilon^2}\right)K_{ij}=-\frac{1}{\sqrt{|\nabla u_{\varepsilon}|^2+\varepsilon^2}}.
\end{align*}
A notable feature of elliptic regularization is that the downward translating graph 
 \begin{equation}\label{translating}
\wt\Sigma^{\varepsilon}_t:=\graph\Bigl(\dfrac{u_{\varepsilon}}{\varepsilon}-\frac{t}{\varepsilon}\Bigr)
\end{equation}
solves the classical evolution $(*)$ in the product manifold $(M\times\mathbb{R},\bar{g}:=g\oplus dz^2)$, where we extend the given data $K$ to be parallel in the $z$-direction. Furthermore, this elliptic regularization problem sheds new light on the study of Jang's equation (\ref{Jang}), since the rescaled function $\hat{u}_{\varepsilon}:=\frac {u_{\varepsilon}}{{\varepsilon}}$ solves 
\[\tag{$*_{\hat{\varepsilon}}$}
\mbox{div}\left(\dfrac{\nabla \hat{u}_{\varepsilon}}{\sqrt{|\nabla \hat{u}_{\varepsilon}|^2+1}}\right)-\left(g^{ij}-\dfrac{\nabla^i\hat{u}_{\varepsilon}\nabla^j\hat{u}_{\varepsilon}}{|\hat{u}_{\varepsilon}|^2+1}\right)K_{ij}=-\dfrac{1}{\varepsilon\sqrt{|\nabla \hat{u}_{\varepsilon}|^2+1}},
\]
which can be interpreted as equation (\ref{Jang}) with a gradient regularization term. Analogous to the situation for Jang's equation, the scalar term $g^{ij}K_{ij}$ obstructs the existence of a supremum estimate for a solution of $(*_{\hat \varepsilon})$. In order to overcome this problem, we introduce the capillarity regularization term studied by Schoen and Yau in \cite{SY81}. Subsequently, we find that when taking the limit of this capillarity regularization term, the solution $\hat{u}_{\varepsilon}$ of $(*_{\hat{\varepsilon}})$ blows up to infinity over a MOTS. 


The main results of this work are summarized in the following theorem.
\begin{Theorem}\label{MOTS}
Let $(M^{n+1},g,K)$ be an initial data set for a space-time, and let $\partial\Omega_{\out}=\Sigma_{\out}$ denote the  outermost MOTS in $M$. Let $\Omega$ be a smooth domain in $M$ with $\Omega_\out\subset\Omega$, and whose boundary, $\partial\Omega$, is a mean convex closed and bounded outer trapped embedded hypersurface in $M$. Then for $2\le n\le 6$ the following hold: 
\begin{enumerate}[(i)]
\item Let  $\lambda=\max_i\{|\l_i|, \l_i\text{  eigenvalue of  } K\}$. Then, for any $0<\e\le \min\left\{\frac{1}{(n+1)\lambda}, \frac12\right\}$  there exists a  solution $\hat{u}_{\varepsilon}\in C^\infty(\Omega\setminus\Omega_\out)$ of the equation $(*_{\hat{\varepsilon}})$ that is zero on $\partial\Omega$ and blows up to infinity over $\Sigma_\out$,  that is $\lim_{x\to x_0}\hat u_\e(x)=\infty$ for any $x_0\in \Sigma_\out$.
\item There exists a sequence of $\hat{u}_{\varepsilon_k}$ as in (i) with $\e_k\downarrow 0$ such that $u_{\e_k}\to u$ in $C^0(\Omega_1\cup\partial \Omega)$, where $u\in C^{0,1}(\Omega_1\cup\partial\Omega)$  and $\Omega_1\subset \Omega\setminus \Omega_\out$ is such that $\partial\Omega\subset\partial\Omega_1$ and $\partial^* (\Omega\setminus\Omega_1)$ is a {\it generalized} MOTS (see Definition~\ref{weakMOTS} and Remark~\ref{weakrmk}).
\end{enumerate}
\end{Theorem}\label{main}
\begin{Remark}\label{thmrmk} We will call a function  $u$ as in (ii) of Theorem~\ref{MOTS}  a weak solution of $(**)$ and its level sets  $(\Sigma_t=\{u=t\})_{t\ge 0}$  a weak solution  of $(*)$  (see Definition~\ref{weak}). Theorem~\ref{MOTS} (ii) then states that there exists a weak solution  of $(*)$ with initial condition $\Sigma_0=\partial\Omega$, $(\Sigma_t)_{t\ge 0}$,  that converges to a generalized MOTS that lies outside the outermost MOTS. Note also that the fact that the outermost MOTS has the form $\Sigma_\infty=\partial\Omega_\infty$, where $\Omega_\infty$ is an open set, is not an assumption---this is always the case with $\Omega_\infty$ being the union of all weakly outer trapped sets in $M$, that is open sets with  weakly outer trapped boundary (i.e. satisfying $\theta^+\le 0$), as is shown in \cite{AM09}.
\end{Remark}

\begin{Remark}
 In case $\partial \Omega_\out=\emptyset$, then Theorem~\ref{MOTS} still holds with the functions $\hat{u}_{\varepsilon}$, as in (i), being defined over all of $\Omega$ (see Theorem \ref{kto0} (iii)).
 \end{Remark}
 \begin{Definition}\label{weakMOTS} Let $E\subset M$ be a finite perimeter set. We will say that the reduced boundary of $E$, $\partial^*E$,  is a generalized MOTS if the following hold 
 \begin{itemize}
 \item $\mu_E=\H^n\res\partial^*E$ carries a generalized mean curvature vector $\vec H$  and
 \item For $\H^n$-a.e. point on $\partial^*E$
 \[\vec H+ P\nu=0\,,\,\,\text{where }P=\nu^i\nu^jK_{ij}\]
 and $\nu$ is the measure theoretic outer pointing unit normal to $\partial^*E$.
  \end{itemize}
 (See  \cite[Definition 16.5 and \S14]{LSgmt} for precise definitions of $\partial^*E$, $\vec H$ and $\nu$).
 \end{Definition}
 \begin{Remark}[on Definition~\ref{weakMOTS}] \label{weakrmk}
 If $\partial^* E$ is a generalized MOTS, then by Allard's regularity theorem \cite{Al} (see also \cite{Alex}) we infer that, away from a set of $\H^n$-measure zero, $\partial^*E$ is a $C^{1,\a}$ hypersurface for any $\a\in (0,1)$. This implies that locally (away from a set of $\H^n$-measure zero) it is the graph of a function that satisfies  equation \eqref{Jang} weakly, and using standard PDE methods we obtain that  $\partial^* E$ is smooth, and thus a MOTS in the classical sense, away from a set of $\H^n$-measure zero.
 
 Furthermore, since the mean curvature is bounded on the reduced boundary, we also have that if  $\partial^* E$ is a generalized MOTS then $\H^n(\partial E\setminus\partial^*E)=0$.
 \end{Remark}
 \begin{Remark}The proof of the main theorem, Theorem~\ref{main}, is given in Theorems \ref{kto0}, \ref{blowup}, \ref{weak existence} and \ref{MOTSconv}. 
 
In Section \ref{properties}, we give various properties for the graphs of the functions $\hat u_\e$, $\hat u$, the most important being a minimizing property (see Lemmas \ref{emin},  \ref{eminu}). Furthermore, in addition to (locally) uniform convergence of the functions $\hat u_\e$ to $\hat u$, we obtain  convergence, in the sense of varifolds, of their graphs (see Theorem~\ref{mainconvii}).  \end{Remark}

\noindent{\bf Remarks on further directions.} 

 We believe that weak solutions $(\Sigma_t)_{t\ge0}$ (as in Remark~\ref{thmrmk}) actually  converge to the outermost MOTS. However, as our proof yields only {\it weak} convergence of the $\Sigma_t$'s as $t\to \infty$, we can only deduce that the limit is a {\it generalized} MOTS. If the generalized limit can be shown to be regular, then, as it lies outside the outermost MOTS,  the two must coincide. 
 We believe that it should be possible to adapt techniques from \cite{W00} to show that the level sets $\Sigma_t$ have a better minimization property (than the one-sided minimization property of Lemma~\ref{eminu}) and thus obtain better regularity for the limit. At the end of Section \ref{properties} we discuss this  in greater detail.


\section{The smooth flow}
\noindent This work focuses on the development of a theory of weak solutions to null mean curvature flow, and in this sense does not provide a classical, PDE analysis of $(*)$, except for the following remarks laid out here.

Direct calculation reveals that the null mean curvature flow $(*)$ can be expressed in terms of the Laplace-Beltrami operator $\Delta_{g(t)}$ with respect to the metric $g(t)$ as follows 
\begin{equation*}
 \frac{\partial F}{\partial t}(x,t) = \Delta_{g(t)}F(x,t)-g(t)^{ij}K_{ij}.\\
\end{equation*}
Null mean curvature flow is therefore a quasi-linear, weakly parabolic system which inherits many properties from and indeed formally resembles the standard heat equation (plus a lower order term). It arises as the steepest descent flow of area plus bulk energy $P$, since
\begin{equation*}
\frac{d}{dt}\left(|\Sigma_t|+\int_{V_t}P dV\right)=-\int_{\Sigma_t} H(H+P) + P(H+P)d\mu=-\int_{\Sigma_t} (H+P)^2 d\mu, 
\end{equation*}
where $V_t$ denotes the volume traced out by the family of hypersurfaces over the time period $[0,t]$. 

The reaction-diffusion system governing the null mean curvature of $\Sigma_t$ is given by
\begin{equation}\label{heH+P}
\begin{split}\frac{\partial}{\partial t}(H+P)=&\Delta(H+P)+(H+P)(|A|^2+\Ric(\nu,\nu))\\
&-(H+P)(\nabla_{\nu}\text{tr}_MK-(\nabla_{\nu}K)(\nu,\nu))-2D_i(H+P)K_{i\nu}.
\end{split}
\end{equation}
If, for example, $\Sigma_0$ is closed, the cubic reaction term on the right-hand side guarantees singularity formation in finite time, analogous to the situation for mean curvature flow. This motivates the development of a weak solution to extend the evolution beyond the classical singular time.

\vskip 0.1 true in
\textbf{Monotonicity Formula}
We do not study the classification of singularities of the evolution by null mean curvature in this paper, however it is interesting to point out that the heat kernel monotonicity formula for mean curvature flow, proven by Huisken in \cite{Hu90}, generalizes to the null mean curvature flow. By the work of Hamilton \cite{Ha1}, it is known that Huisken's monotonicity formula generalizes to mean curvature flow on a manifold. The monotonicity formula we present here is very close to that of  Hamilton's \cite{Ha1}, with the extra complication that one needs to estimate the extra $P$-term (coming from the speed being here $H+P$ instead of $H$). We remark that such an estimate has been carried out also in \cite{W96} for the case of mean curvature flow with additional forces in Euclidean space.

Let $\psi:M\times [0,T)$, for $T>0$ be a positive solution of the backward heat equation on $M\times [0,T)$ \[\frac{\partial\psi}{\partial t}=-\Delta\psi.\]
We prove a monotonicity formula for the integral of the function
\[\phi:=(4\pi (T-t))^\frac12\psi.\]
We have that
\begin{equation*}\frac{d}{dt}\dmt=-\vec H(\vec H+\vec P) d\mu_t=-H(H+P) d\mu_t
\end{equation*}
and 
\[\frac{d\psi}{dt}=\frac{\partial\psi}{\partial t}+\nabla\psi\cdot (\vec H+\vec P)=-\Delta \psi+\nabla^\bot\psi\cdot (\vec H+\vec P),\]
where $\nabla^\bot= \nabla\psi\cdot \nu$, $\vec H=-H\nu$ and $\vec P=-P\nu$. Hence,
\[
\begin{split}\frac{d}{dt} \int_{\Sigma_t}\phi\dmt&= (4\pi (T-t))^\frac12\int_{\Sigma_t}\bigg(-\frac{\psi}{2(T-t)}-\psi \vec H(\vec H+\vec P)\\
&\hspace{5cm}-\Delta \psi +\nabla^\bot\psi\cdot (\vec H+\vec P)\bigg)\dmt.\end{split}
\]
Since
\[\begin{split}\Delta_{\Sigma_t}\psi&=\dvg_{\Sigma_t}(\nabla^{\Sigma_t}\psi)=\dvg_{\Sigma_t}(\nabla\psi)-\dvg_{\Sigma_t}(\nabla^\bot\psi)= \dvg_{\Sigma_t}(\nabla\psi)+\nabla^\bot\psi\cdot \vec H\\
&=\Delta\psi- D^2\psi(\nu, \nu)+ \nabla^\bot\psi\cdot \vec H,
\end{split}\]
we find that
\[\begin{split}\frac{d}{dt}&\int_{\Sigma_t}\phi\dmt\\
&= (4\pi (T-t))^\frac12\int_{\Sigma_t}\biggl(-\frac{\psi}{2(T-t)}-\psi \vec H(\vec H+\vec P)-\Delta_{\Sigma_t} \psi-D^2\psi(\nu,\nu)\\
& \hspace{7.3cm}+2\nabla^\bot\psi\cdot \vec H+\nabla^\bot\psi\cdot\vec P\biggr)\dmt\\
&=(4\pi (T-t))^\frac12\int_{\Sigma_t}\biggl(-\psi\left(\vec H+\vec P-\frac{\nabla^\bot \psi}{\psi}\right)^2+\psi  \vec P\left( \vec H+\vec P-\frac{\nabla^\bot\psi}{\psi}\right)\\
&\hspace{4cm}-\frac{\psi}{2(T-t)}-\Delta_{\Sigma_t} \psi-D^2\psi(\nu,\nu)+\frac{|\nabla^\bot\psi|^2}{\psi} \biggr)\dmt.\end{split}\]
Define now
\[Q(\psi)= \frac{\psi}{2(T-t)}+D^2\psi(\nu,\nu)- \frac{|\nabla^\bot\psi|^2}{\psi}\]
and let $P_0=\sup_M|P|$. Noticing that $\int_{\Sigma_t}\Delta_{\Sigma_t}\psi\dmt=0$ and applying the Cauchy--Schwarz inequality, we obtain
\[\begin{split}
\frac{d}{dt}\int_{\Sigma_t}\phi\dmt &\le  \frac12\int_{\Sigma_t}-\phi\left(\vec H+\vec P-\frac{\nabla^\bot \psi}{\psi}\right)^2\dmt\\
&\quad- (4\pi (T-t))^\frac12\int_{\Sigma_t}Q(\psi)\dmt+\frac{P_0^2}{2}\int_{\Sigma_t}\phi\dmt.
\end{split}\]
Note that $Q(\psi)$ is the quantity that appears in Hamilton's Harnack matrix inequality \cite{Ha}, and in the special case where $\nabla \Ric=0$ and the sectional curvatures of $M$ are non-negative, this implies that $Q(\psi)\ge 0$.  In general, we find that there exist constants $B,C$ depending only on $M$ such that
\[-Q(\psi)\le C\psi\left(1+\log\left(\frac{B}{(T-t)^\frac{n+1}{2} \psi}\right)\right).\]
Using the inequality $x(1+\log(y/x))\le 1+x\log y$ (see \cite{Ha1}) we obtain
$-Q(\psi)\le C(1+\psi\log (B(T-t)^{-\frac{n+1}{2}})$ and thus
\[\begin{split}
\frac{d}{dt}\int_{\Sigma_t}\phi\dmt &\le\frac{P_0^2}{2}\int_{\Sigma_t}\phi\dmt + C\log \left(\frac{B}{(T-t)^{\frac{n+1}{2}}}\right)\int_{\Sigma_t}\phi\dmt\\
&\quad+C(4\pi (T-t))^\frac12|\Sigma_t|.
\end{split}\]
Seting
\[\zeta(t)=(T-t)\left(\frac{P_0^2}{2}+C\frac{n+1}{2}+C\log \left(\frac{B}{(T-t)^{\frac{n+1}{2}}}\right)\right), \]
we observe that 
\[\frac{d\zeta}{dt}=-\frac{P_0^2}{2}-C\log \left(\frac{B}{(T-t)^{\frac{n+1}{2}}}\right)\]
and thus, 
\[\frac{d}{dt}\left(e^{\zeta(t)}\int\phi d\mu_t\right)\le C(4\pi (T-t))^\frac12|\Sigma_t|.\]


\section{Level-set description and elliptic regularization}
In this section we employ the level-set approach, which transforms the hypersurface evolution equation $(*)$ into a degenerate elliptic equation for a scalar level-set function. We then define the elliptic regularized problem that we will use to prove existence of weak solutions in a later section.

\vskip 0.1 true in
\textbf{Level-set formulation.}
Assume that the evolving hypersurfaces are given by the level sets of a scalar function $u:M\to\mathbb{R}$ via 
\begin{equation*}
E_t:=\{x:u(x)>t\},\quad\quad\Sigma_t:=\partial E_t,
\end{equation*}
where $E_0=\Omega$ and $\partial \Omega$ is an outer untrapped closed and bounded mean convex hypersurface, so that $(H+P)|_{\partial \Omega}>0$ and $H_{\partial \Omega}>0$. Then, wherever $u$ is smooth and $\nabla u\neq0$, the (outward) normal vector to $\Sigma_t$ is given by $\nu=-\dfrac{\nabla u}{|\nabla u|}$ and the boundary value problem
\[\tag{$**$} \,\,
\left\lbrace
\begin{aligned}
 \mbox{div}\left( \frac{\nabla u}{|\nabla u|}\right)-\left(g^{ij}-\dfrac{\nabla^iu\nabla^ju}{|\nabla u|^2}\right)K_{ij} &=-\dfrac{1}{ |\nabla u|} , \\
           u\Big|_{\partial \Omega} \,\, &= \,\, 0,
\end{aligned}
\right.
\]
describes the evolution of the level sets of $u$ by null mean curvature. In particular, the left-hand side represents the negative null mean curvature of $\Sigma_t$ and the right-hand side is the speed of the family of level sets in the outward unit normal direction $\nu$.  

\vskip 0.1 true in
\textbf{Elliptic regularization.} As a first step towards establishing existence of weak solutions to the degenerate elliptic problem $(**)$, we study solutions of the following strictly elliptic equation, for $\varepsilon>0$
\[\tag{$*_{\varepsilon}$}
\left\lbrace
\begin{aligned}
 \mbox{div}\left(\dfrac{\nabla u_{\varepsilon}}{\sqrt{|\nabla u_{\varepsilon}|^2+\varepsilon^2}}\right)-\left(g^{ij}-\dfrac{\nabla^iu_{\varepsilon}\nabla^ju_{\varepsilon}}{|\nabla u_{\varepsilon}|^2+\varepsilon^2}\right)K_{ij}&=-\dfrac{1}{\sqrt{|\nabla u_{\varepsilon}|^2+\varepsilon^2}},\\
u_{\varepsilon}\big|_{\partial\Omega}\,\,&=\,\,0.
\end{aligned}
\right.
\]
Then, rescaling $(*_{\varepsilon})$ via $u_{\varepsilon}:=\varepsilon\hat{u}_{\varepsilon}$, we obtain
\[\tag{$*_{\hat{\varepsilon}}$}
\mbox{div}\left(\dfrac{\nabla \hat{u}_{\varepsilon}}{\sqrt{|\nabla \hat{u}_{\varepsilon}|^2+1}}\right)-\left(g^{ij}-\dfrac{\nabla^i\hat{u}_{\varepsilon}\nabla^j\hat{u}_{\varepsilon}}{|\hat{u}_{\varepsilon}|^2+1}\right)K_{ij}=-\dfrac{1}{\varepsilon\sqrt{|\nabla \hat{u}_{\varepsilon}|^2+1}}.
\]
Here we interpret the left-hand side as the negative null mean curvature $-(H+P)$ of the hypersurface $\graph \hat{u}_{\varepsilon}$ in the product manifold 
\begin{equation}\label{product}
(M^{n+1}\times\mathbb{R},\bar{g}),\quad\quad\bar{g}:=g\oplus dz^2,
\end{equation}
with respect to the upward pointing unit normal $\hat{\nu}_{\varepsilon}:=\dfrac{(-\nabla\hat{u}_{\varepsilon},1)}{\sqrt{1+|\nabla\hat{u}_{\varepsilon}|^2}}$ of the graph, where we extend the given data $K$ to be constant in the $z$-direction. We also extend the unit  normal $\hat{\nu}_{\varepsilon}$ so that it is constant in the $z$-direction. On the right-hand side of $(*_{\hat{\varepsilon}})$ we have 
\begin{equation}\label{gradterm}
-\dfrac{1}{\varepsilon\sqrt{|\nabla \hat{u}_{\varepsilon}|^2+1}}=-\frac{1}{\varepsilon}\langle\tau_{n+2},\hat{\nu}_{\varepsilon}\rangle,
\end{equation}
where $\tau_{n+2}$ is the unit vector in the $z$-direction. Thus, $(*_{\varepsilon})$ has the geometric interpretation that the downward translating graph
\begin{equation} \label{graph}
 \wt\Sigma^{\varepsilon}_t:=\graph \left(\hat{u}_{\varepsilon}-\dfrac{t}{\varepsilon}\right),
\end{equation}
solves $(*)$ smoothly in $\Omega\times\mathbb{R}$. This is equivalent to the statement that the function 
\begin{equation*}U_{\varepsilon}(x,z):=u_{\varepsilon}(x)-\varepsilon z,\quad\quad(x,z)\in\Omega\times\mathbb{R},
\end{equation*}
solves $(**)$ in $\Omega\times\mathbb{R}$, since $U_{\varepsilon}$ is the time-of-arrival function for the solution $\wt\Sigma^{\varepsilon}_t$, that is
\begin{equation}\label{ellip}
 \wt\Sigma^{\varepsilon}_t=\{U_{\varepsilon}=t\}.
\end{equation}
We conclude that elliptic regularization allows one to approximate solutions of $(**)$ by smooth, noncompact, translating solutions of $(*)$ one dimension higher. 

\section{Elliptic regularization and Jang's equation}\label{ellregsection}
In fact, $(*_{\hat{\varepsilon}})$ has the further interpretation as Jang's equation \eqref{Jang}
with the gradient regularization term given by (\ref{gradterm}). Equation (\ref{Jang}) was introduced by Jang in  \cite{J78} to generalize Geroch's \cite{G73} approach to proving the positive mass theorem from the time symmetric case to the general case. Jang noted however, that the equation cannot be solved in general, leaving the question of existence and regularity of solutions open. The analytical difficulty is the lack of an a-priori estimate for $\sup_\Omega |u|$ due to the presence of the zero order term $\tr_M(K)$. For this reason, it is necessary to introduce a regularization term to $(\ref{Jang})$ in order to prove existence of solutions. 

In \cite{SY81}, Schoen and Yau introduce a positive capillarity regularization term that provides a direct supremum estimate via the maximum principle, and study existence of solutions to the following regularized Jang's equation
\begin{align}
\label{SY}\left(g^{ij}-\dfrac{\nabla^iu_{\k}\nabla^ju_{\k}}{|\nabla u_{\k}|^2+1}\right)\left(\dfrac{\nabla_i\nabla_ju_{\k}}{\sqrt{|\nabla u_{\k}|^2+1}}+K_{ij}\right)=\k u_{\k}\quad&\text{on }M,\\
\notag u_{\k}\to0\quad\,\,\,\,&\text{as }|x|\to\infty.
\end{align}
It is interesting to compare the following three approaches to regularizing Jang's equation:
\begin{enumerate}[(i)]
\item A capillarity regularization term as in (\ref{SY}) above.
\item The gradient regularization term $\frac{-1}{\varepsilon\sqrt{1+|\nabla \hat{u}_{\varepsilon}|^2}}$ in $(*_{\hat{\varepsilon}})$, the (rescaled) elliptic regularization problem for null mean curvature flow in this work.
\item The gradient regularization term $\varepsilon\sqrt{1+|\nabla \hat{u}_{\varepsilon}|^2}$, which appears in the (rescaled) elliptic regularization problem 	for the evolution by inverse null mean curvature, studied in \cite{KM13}.
\end{enumerate} 

\noindent In particular, the gradient function $\sqrt{1+|\nabla \hat{u}_{\varepsilon}|^2}$ appearing in cases (ii) and (iii) is related to the vertical component of the upward pointing unit normal $\hat{\nu}_{\varepsilon}$ of $\graph\hat{u}_{\varepsilon}$ via 
\begin{equation}
\langle\tau_{n+2}, \hat{\nu}_{\varepsilon}\rangle=\frac{1}{\sqrt{1+|\nabla\hat{u}_{\varepsilon}|^2}}.
\end{equation} 
This means that the  graphs $\Sigma^{\varepsilon}_t:=\graph\left(\hat{u}_{\varepsilon}-\dfrac{t}{\varepsilon}\right)$ of the function $\hat u_\e$ solving the regularized Jang's equations described by cases (ii) and (iii) above have the additional property of being smooth, translating solutions---one dimension higher, in $M^n\times\mathbb{R}$---of the evolution by null mean curvature, and inverse null mean curvature, respectively.

 In this way, (\ref{SY}) can be viewed as a static, elliptic PDE approach to studying solutions to Jang's equation, as opposed to the evolutionary, parabolic PDE approach as given by the elliptic regularized equation for null mean curvature flow in this work, and the evolution by inverse null mean curvature in \cite{KM13}. The advantage of a parabolic approach is that it not only proves existence of MOTS, but also gives a good idea of what they actually look like---in particular by providing a constructive method for the numerical modeling of solutions.

It turns out however that the gradient regularization terms in (ii) and (iii) are not sufficient on their own to overcome the problem associated with the zero order term $\tr_M(K)=g^{ij}K_{ij}$ in Jang's equation. For the evolution by inverse null mean curvature, as in case (iii), the term $\tr_M(K)$ obstructs the existence of a lower barrier at the inner boundary, and it is necessary to restrict to space time initial data sets $(M, g, K)$ such that $\tr_M(K)\geq 0$ in order to prove existence of solutions to the regularized Jang's equation. 
In the case of null mean curvature flow studied here, we introduce the capillarity regularization term of Schoen and Yau in order to obtain the required supremum estimate to solve $(*_{\hat{\varepsilon}})$.

\vskip 0.1 true in

\noindent\begin{large}\textbf{Adding a capilarity regularization term.}\end{large}\\
As discussed above, in order to overcome the difficulties associated with the zero order term $g^{ij}K_{ij}$, we add the capillarity regularization term to $(*_{\hat{\varepsilon}})$ and study solutions $\hat{u}=\hat{u}_{\varepsilon,\kappa,s}$ of the following problem
\[\tag{$*_{\hat{\varepsilon},\kappa,s}$} \,\,
\left\lbrace
\begin{aligned}
 \mbox{div}\left(\dfrac{\nabla \hat{u}}{\sqrt{|\nabla \hat{u}|^2+1}}\right)-s\left(g^{ij}-\dfrac{\nabla^i\hat{u}\nabla^j\hat{u}}{|\nabla \hat{u}|^2+1}\right)K_{ij}&=\dfrac{-s}{\varepsilon\sqrt{|\nabla \hat{u}|^2+1}}+\kappa\hat{u},\\
\hat u\big|_{\partial\Omega}\,\,&=\,\,0
\end{aligned}
\right.
\]
for $\varepsilon>0$, $\kappa>0$, $s\in[0,1]$ and $\Omega$  an open and bounded set in $M$. The parameter $s$ has been added here to simplify the proof of existence using the implicit function theorem in Lemma~\ref{kappaexists} below. Once existence of solutions of $(*_{\hat{\varepsilon},\kappa,s})$ has been established, we may fix $s=1$ and take the limit as $\kappa$ goes to zero to obtain existence of solutions to $(*_{\hat{\varepsilon}})$. In the study of the regularized Jang's equation (\ref{SY}) in \cite{SY81}, the supremum and gradient  estimates blow up when  $\kappa\to0$, and Harnack-type estimates imply that the boundary of the blowup set is a MOTS in $(M,g)$. We will observe below that the same blow-up behaviour arises for solutions of $(*_{\hat{\varepsilon}})$. 

 We now derive the required a-priori estimates for $(*_{\hat{\varepsilon},\kappa,s})$.
\begin{Lemma}[Supremum estimate]\label{supest}
 Let  $\lambda=\max_i\{|\l_i|, \l_i\text{  eigenvalue of  } K\}$. For any $\e\le \frac{1}{(n+1)\lambda}$ solutions $\hat{u}$ of $(*_{\hat{\varepsilon},\kappa,s})$ satisfy the  estimate
\begin{equation*}\label{supbnd}
0\leq \hat{u}\leq \frac{2}{\varepsilon\kappa}.
\end{equation*}

\end{Lemma}
\begin{proof}
Since $\hat u|_{\partial\Omega}=0$, either $\hat{u}\leq0$ or $\hat{u}$ has an interior maximum. At an interior maximum point we have 
\begin{align*}
\max_{\Omega}\kappa \hat{u}= g^{ij}\hat{u}_{ij}-sg^{ij}K_{ij}+\frac{s}{\varepsilon}\leq (n+1)\lambda+\frac{1}{\varepsilon}.
\end{align*}
Since, for $\varepsilon\leq\frac{1}{(n+1)\lambda}$, zero is a subsolution of $(*_{\hat{\varepsilon},\kappa,s})$  we find
\begin{equation*}
0\leq \hat{u}\leq \frac{(n+1)\lambda}{\kappa}+\frac{1}{\varepsilon\kappa}\le \frac{2}{\varepsilon\kappa}.
\end{equation*}
\end{proof}
\begin{Lemma}[Gradient estimate]\label{interiorforkappa}
 For any $\e\le \frac12$ solutions $\hat{u}$ of  $(*_{\hat{\varepsilon},\kappa,s})$  satisfy the estimate
\begin{equation*}
\sup_{\Omega}|\nabla \hat{u}|\leq \exp(\eta\sup_{\Omega} \hat{u})\cdot\sup_{\partial\Omega}\left(\frac{1}{\varepsilon}+\sqrt{1+|\nabla \hat{u}|^2}\right),
\end{equation*}
where $\eta$ is a constant that depends only on the initial data, in fact $\eta=\eta(n,\Ric,\|K\|_{C^1})$.
\end{Lemma}
\begin{proof} For the gradient function $v(x,f(x)):=\sqrt{1+|\nabla f(x)|^2}$ of a hypersurface  $N=\graph f$ we have
\begin{equation}\label{jacobi}
\Delta^N v=\frac{2}{v}|\nabla^N v|^2-v^2\bar g(\nabla^N H,\tau)+v|A|^2+v\left(1-\frac{1}{v^2}\right)\Ric(\gamma,\gamma),
\end{equation}
where $\tau=\frac{\partial}{\partial z}$ is the unit vector pointing in the upward, $\mathbb{R}$, direction of $M\times\mathbb{R}$, $\nu$ is the upward pointing unit normal to $N=\graph f$, $H$ and $A$ are the mean curvature and the second fundamental form of $N$,  $\gamma:=\frac{\text{pr}_{TM}\nu}{|\text{pr}_{TM}\nu|}$ in case $\nu\ne \tau$ and zero otherwise,  and $\Ric=\Ric_M$ is the Ricci curvature of $M$. For details of the derivation of \eqref{jacobi} see \cite[(13)]{S08}. Recall also that $\bar g$ is the metric in the product manifold $M\times\R$ as defined in \eqref{product}.  We follow the general approach in \cite[Lemma 3.2]{S08} to show that we can obtain a gradient bound given an a-priori height bound and compute $\Delta^N(wv)$, where
 $w(x,z):=\exp(-\eta z)$, for $(x,z)\in M\times\R$ and $\eta>0$ a constant to be chosen later. For the function $w$ we have
\[\nabla^N w=-\eta w\left(\tau-\frac1v\nu\right)\,\text{ and }\,\Delta^N w=\eta^2\left(1-\frac{1}{v^2}\right)w+\eta\frac{H}{v}w,\]
and combining these with \eqref{jacobi}, we obtain
\begin{equation}\label{jacobi2}
\begin{split}
\Delta^N(wv)=\frac{2}{v}\bar g(\nabla v,\nabla (wv))+wv\Biggl(&|A|^2+\left(1-\frac{1}{v^2}\right)\Ric(\gamma,\gamma)\\
&+\eta^2\left(1-\frac{1}{v^2}\right)+\eta\frac{H}{v}-v\bar g(\nabla^N H,\tau)\Biggr).
\end{split}
\end{equation}
In order to obtain a contradiction, define $C_1:=\sup_{\partial\Omega}\varepsilon\sqrt{1+|\nabla \hat{u}|^2}$ and assume 
\begin{equation}\label{contradiction1}
\sup_{\Omega}(\exp(-\eta \hat{u})\varepsilon\sqrt{1+|\nabla \hat{u}|^2})>\max\{C_1,1\},
\end{equation}
which must be attained at an interior point $x_0$. Letting $N=\graph\hat{u}$, equation $(*_{\hat{\varepsilon},\kappa, s})$ implies that 
\begin{equation}\label{floweqn}
H+sP=\dfrac{s}{\varepsilon v}-\kappa\hat{u},
\end{equation}
where $H+P$ is the null mean curvature of $N$. 
Now, using the expression for $\nabla^N w$,  we find 
\begin{align*}
wv^2\bar g(\nabla^N H,\tau)=&-\frac{s}{\varepsilon}\bar g(\nabla^N (wv),\tau)+\frac{sv}{\varepsilon}\bar g(\nabla^N w,\tau)-wv^2s\bar g(\nabla^N P,\tau)\\
&-wv^2\kappa \bar g(\nabla^N \hat u,\tau)\\
=&-\frac{s}{\varepsilon}\bar g(\nabla^N (wv),\tau)-\frac{s}{\varepsilon}\eta wv\left(1-\tfrac{1}{v^2}\right)-wv^2s\bar g(\nabla^N P,\tau)\\
&-wv^2\kappa \bar g(\nabla^N  \hat u,\tau).
\end{align*}
Note that $\nabla^NP=\ov\nabla P-\bar g(\ov\nabla P,\nu)\nu$ (where $\ov\nabla=\nabla^{M\times\R}$),  $K$ (as well as $\nu$)  is extended trivially in the $\tau$ direction so that $\bar g(\ov\nabla P,\tau)=0$, and $\bar g(\tau,\nu)=\frac{1}{v}$. Using these, we obtain
\begin{equation}\label{MNest}
\begin{split}
vs\bar g(\nabla^N P,\tau)=-s\bar g(\ov\nabla P,\nu)&\geq-c(1+|A|),\\
v\kappa \bar g(\nabla^N  \hat{u},\tau)=-\kappa \bar g(\ov\nabla \hat{u},\nu)&=\frac{\kappa|\nabla \hat{u}|^2}{v}\geq0,
\end{split}
\end{equation}
where $c=c(n, \|K\|_{C^1})\ge1$, so that, using the Cauchy--Schwarz inequality, we have
\begin{equation}\label{nablaHest}
\begin{split}
wv^2\bar g(\nabla^N H,\tau)\leq-\frac{s}{\varepsilon}\bar g(\nabla^N (wv),\tau)&-\frac{s}{\varepsilon}\eta wv\left(1-\frac{1}{v^2}\right)+2c^2wv+ wv\frac{|A|^2}{2}.
\end{split}
\end{equation}
At a maximum point $x_0$, where $\Delta^N(wv)\leq0$ and $\nabla^N(wv)=0$, (\ref{jacobi2}) becomes
\begin{align*}
0\geq&|A|^2 +\left(1-\frac{1}{v^2}\right)\Ric(\gamma,\gamma)+\eta^2\left(1-\frac{1}{v^2}\right)+\eta\frac{H}{v}-v\bar g(\nabla^N H,\tau),
\end{align*}
where the constant $c=c(n, \|K\|_{C^1})$ is the constant from \eqref{MNest}. Using (\ref{floweqn}),  \eqref{nablaHest} and Lemma \ref{supest}, we obtain
\[
\begin{split}
0\geq \frac{|A|^2}{2}&+\left(1-\frac{1}{v^2}\right)\left(\Ric(\gamma,\gamma)
+\frac{s\eta}{\varepsilon}+\eta^2\right)\\
&+\frac{s\eta}{\varepsilon v^2}-\eta\frac{\|K\|_{C^0}}{v}-\frac{2\eta}{\e v}-2c^2.
\end{split}
\]
By the assumption (\ref{contradiction1}), we find that $v(x_0)>\frac{1}{\varepsilon}$ and thus $(1-\frac{1}{v^2})>\frac{1}{2}$ when $\varepsilon\leq\frac{1}{2}$. Then, the above becomes
\begin{equation*}
0\geq \frac{\eta^2}{2}+\frac{1}{2}\Ric(\gamma,\gamma)-2c^2-\eta\left(\frac12\|K\|_{C^0}+2\right),
\end{equation*}
where the constant $c=c(n, \|K\|_{C^1})$ is the constant from \eqref{MNest}, and setting $\eta=\eta(n,\Ric,\|K\|_{C^1})$ large enough so that the right-hand side is strictly positive leads to a contradiction and thus  hypothesis (\ref{contradiction1}) is false.
\end{proof}

\begin{Lemma}[Boundary gradient estimate]\label{boundaryforkappa}
Assume that  $\partial\Omega$ is smooth,  strictly mean convex and outer untrapped with respect to the outward pointing unit normal. 
Then, 
solutions $\hat{u}$ of $(*_{\hat{\varepsilon},\kappa,s})$ satisfy the estimate
\begin{equation*}
\sup_{\partial\Omega}|\nabla \hat{u}|\leq C(\|K\|_{C_0},\varepsilon,\theta^+_{\partial\Omega}),\end{equation*}
where we recall that $\theta^+_{\partial\Omega}$ is the null mean curvature of ${\partial\Omega}$ with respect to its future directed outward pointing null vector field.
\end{Lemma}
\begin{proof}
Since $\partial\Omega$ is strictly mean convex and outer untrapped with respect to the outward pointing unit normal (so that on $\partial \Omega$ $H+sP>0$ for $s\in[0,1]$) we can use the classical barrier construction of Serrin, as presented in   \cite[Theorem 14.6]{GT01}, to obtain the desired boundary gradient estimate. Since equation $(*_{\hat{\varepsilon},\kappa,s})$ is expressed in terms of the geometry of $\graph\hat{u}$, in order to utilize the outer untrapped condition of the boundary $\partial\Omega$, we re-write it instead in terms of the geometry of the individual level sets of $\hat{u}$. To this end, we multiply $(*_{\hat{\varepsilon},\kappa,s})$ by $v^3=\sqrt{1+|\nabla\hat{u}|^2}^3$ to obtain
\begin{align*}
Q(\hat{u}):=&
(1+|\nabla\hat{u}|^2)\left(g^{ij}-\dfrac{\nabla^i\hat{u}\nabla^j\hat{u}}{|\nabla \hat{u}|^2+1}\right)\nabla_{ij}\hat{u}\\
&-s\sqrt{1+|\nabla\hat{u}|^2}^3\left(g^{ij}-\dfrac{\nabla^i\hat{u}\nabla^j\hat{u}}{|\nabla \hat{u}|^2+1}\right)K_{ij}\\
&+\frac{s}{\varepsilon}(1+|\nabla \hat{u}|^2)-\kappa\hat{u}\sqrt{1+|\nabla \hat{u}|^2}^3,
\end{align*}
and decompose it, as in  \cite[(14.43)]{GT01}, into the following coefficients 
\begin{align}\label{levelset}
Q\hat u&=(\Lambda a^{ij}_{\infty}+a^{ij}_0)\nabla_{ij}\hat{u}+|\nabla\hat{u}|\Lambda b_{\infty}+b_0=0,
\end{align} 
where 
\[a^{ij}_{\infty}(x,z,p)=a^{ij}_{\infty}\left(x,\frac{p}{|p|}\right)=\left(g^{ij}-\frac{p^ip^j}{|p|^2}\right),\,\,a^{ij}_0=\frac{p^ip^j}{|p|^2},\,\,\Lambda=1+|p|^2,\]
\[b_{\infty}(x,z,p)=b_{\infty}\left(x,z,\frac{p}{|p|}\right)=-sK_{ij}\left(g^{ij}-\frac{p^ip^j}{|p|^2}\right)-\kappa z\]
and
\[b_0=-\kappa z\frac{\Lambda}{\Lambda^{1/2}+|p|}-sK_{ij}\left(g^{ij}\frac{\Lambda}{\Lambda^{1/2}+|p|}+\frac{p^ip^j}{|p|}\frac{\Lambda^{1/2}}{\Lambda^{1/2}+|p|}\right)+\frac{s}{\varepsilon}\Lambda.\]
Then
\begin{equation*}a^{ij}_{\infty}\left(x,\frac{\nabla\hat{u}}{|\nabla\hat{u}|}\right)\frac{\nabla_{ij}\hat{u}}{|\nabla\hat{u}|}+b^{ij}_{\infty}\left(x,u,\frac{\nabla\hat{u}}{|\nabla\hat{u}|}\right)=-(H+sP)-\kappa \hat{u}, 
\end{equation*}
where here $H+P$ is the null mean curvature of the {\it level sets} of $\hat u$ with respect to the outward pointing unit normal. We see that $b_{\infty}$ is non-increasing in $z$, and also that the correction terms $a_0$ and $b_0$, that arise when considering the curvature of the level sets instead of the graph, are of the order required by the structure condition (14.50) (see also (14.59)) of  \cite[Theorem 14.6]{GT01}. That is, $a_0^{ij}=o(\Lambda)$ and $b_0=o(|p|\Lambda)$ as $|p|\to\infty$. Furthermore, since $\partial \Omega$ is outer untrapped, we see that the boundary curvature condition (14.51) of  \cite[Theorem 14.6]{GT01} is also satisfied since $H-b_{\infty}=H+sP> 0$ at all points on the boundary $\partial\Omega$.  \cite[Theorem 14.6]{GT01} can then be applied, which implies the existence of an upper barrier at any boundary point, and which  depends on the mean curvature of the boundary, $K$ and the supremum bound of of $\hat{u}$ (given in Lemma~\ref{supest}). This finishes the proof of the lemma.
\end{proof}

\begin{Lemma}[Existence for $(*_{\hat{\varepsilon},\kappa,s})$]\label{kappaexists}
Let $(M^{n+1},g,K)$ be an initial data set, $\lambda=\max_i\{|\l_i|, \l_i\text{  eigenvalue of  } K\}$ and $\partial\Omega$  a smooth, strictly mean convex and  outer untrapped hypersurface in $M$. Then, for any $\e\le \min\left\{\frac{1}{(n+1)\lambda},\frac12\right\}$, $\k>0$, $s\in[0,1]$ and $\a\in(0,1)$ there exists a  solution $\hat{u}\in C^{2,\a}(\ov{\Omega})$ of $(*_{\hat{\varepsilon},\kappa,s})$.
\end{Lemma}
\begin{proof} The proof follows that of \cite[Lemma 3]{SY81}.
The idea is to apply the method of continuity to the equation $(*_{\hat{\varepsilon},\kappa,s})$. To this end, fix $\varepsilon\le \min\left\{\frac{1}{(n+1)\lambda},\frac12\right\}$ and $\k>0$,  and define
\begin{align*}
F^s(w):=\dvg\left(\frac{\nabla w}{\sqrt{|\nabla w|^2+1}}\right)&-s\left(g^{ij}-\dfrac{\nabla^iw\nabla^jw}{|\nabla w|^2+1}\right)K_{ij}\\
&+\dfrac{s}{\varepsilon \sqrt{|\nabla w|^2+1}}-\kappa w.
\end{align*}
For any $\a\in (0,1)$,  the map 
\begin{equation*}
F: C^{2,\alpha}_0(\bar{\Omega})\times[0,1]\to C^{\alpha}(\bar{\Omega}) 
\end{equation*}
given by $F(w,s):=F^s(w)$  has the solution $F(0,0)=0$. Let $I$ be the set of $s$ such that $(*_{\varepsilon,\kappa,s})$ has a solution in $C^{2,\alpha}(\bar{\Omega})$ or equivalently the set of $s$ for which there exists $w\in C^{2,\alpha}_0(\bar{\Omega})$ such that $F(w,s)=0$. Then $0\in I$ and we will show that $I$ is an open and closed subset of $[0,1]$, which implies that $I=[0,1]$, thus proving the lemma. To show that $I$ is closed one uses  the a-priori estimates in Lemmas \ref{supest}, \ref{interiorforkappa} and \ref{kappaexists}, standard PDE estimates (which imply `higher' a-priori estimates for a solution; in particular $C^{2,\a}$ for any $\a\in (0,1)$) and the Arzela-Ascoli theorem. To show that $I$ is open, one has to linearize $F^s$ at a solution $f_0$ and apply the inverse function theorem for Banach spaces. For the details of these two claims we refer the reader to \cite[Lemmas 2 and 3]{SY81}   where the arguments on the fact that $I$ is both open and closed are carried out in detail. We remark that the only difference between our case and \cite[Lemma 3]{SY81} is that one has to add the factor $-\frac{\nabla^i f_0}{\e(|\nabla f_0|^2+1)^\frac32}$
in the term $B^i$ that appears in the linearization of $F^s$ at a solution $f_0$ (the notation being here as in \cite[Lemma 3]{SY81}).
\end{proof}
\section{Existence of solutions to $(*_{\varepsilon})$}\label{eexistencesection}
We now consider a fixed $\e\le \min\left\{\frac{1}{(n+1)\lambda},\frac12\right\}$, where, as usual, $\lambda=\max_i\{|\l_i|, \l_i\text{  eigenvalue of  } K\}$, set $s=1$ in $(*_{\hat{\varepsilon},\kappa,s})$ and analyze the limit as $\kappa\to0$ of the graphs $N_{\kappa}=\graph\hat{u}_{\varepsilon,\kappa}$, where $\hat{u}_{\varepsilon,\kappa}$ is a solution of the regularized Jang's equation $(*_{\hat{\varepsilon},\kappa,1})$ (which we denote from now on by (${*_{\hat{\varepsilon},\kappa}}$)), so that
\begin{equation} \tag{${*_{\hat{\varepsilon},\kappa}}$}
\begin{aligned}
 \mbox{div}\left(\dfrac{\nabla \hat{u}_{\varepsilon,\kappa}}{\sqrt{|\nabla \hat{u}_{\varepsilon,\kappa}|^2+1}}\right)=&\left(g^{ij}-\dfrac{\nabla^i\hat{u}_{\varepsilon,\kappa}\nabla^j\hat{u}_{\varepsilon,\kappa}}{|\nabla \hat{u}_{\e, \k}|^2+1}\right)K_{ij}\\
 &-\dfrac{1}{\varepsilon\sqrt{|\nabla \hat{u}_{\varepsilon,\kappa}|^2+1}}+\kappa\hat{u}_{\varepsilon,\kappa}.\end{aligned}
\end{equation}
This equation, along with the bound for $|\kappa \hat u_{\e,\k}|$ provided by Lemma~\ref{supest}, shows that the mean curvature of $N_k$ is uniformly bounded by a constant $C= C(\e)$ independent of $\k$. 

In the language of {\it currents} or of {\it finite perimeter sets} (codimension 1), the bound on the mean curvature implies that $N_k$ is a {\it $C$-minimizing current} (see \cite{DS93}) or a {\it $(C, 1)$-minimal set} (see \cite{MaMi}), i.e.  that
\[\mass(N_k)\le \mass(N_k+\partial Q)+C\mass(Q), \forall (n+1)\text{-current }Q.\]
 Such currents or finite perimeter sets  have been extensively studied  in  \cite{DS93} and \cite{MaMi}, where, among other things, it is shown that they have  compactness and regularity properties similar to those of area minimizing currents. The results in  \cite{DS93, MaMi} are stated for currents (or sets) in Euclidean space, but the codimension 1 results (the case which is of interest to us here) extend to general Riemannian ambient manifolds,  see \cite{Alex}. Applying these results in our case yields the following. For a sequence $\k_i\to 0$ the  sequence  $\{N_{\k_i}\}_{i\in \N}$ has a subsequence which converges (in the sense of currents but also as Radon measures) to a {\it$C$-minimizing current} $N$. Furthermore, in dimensions $n\le 6$, $N$ (and any  {\it $C$-minimizing current}) has no singular set, i.e. it is a $C^{1}$ manifold. We can now prove that the graphs of the sequence  $\{N_{\k_i}\}$ have locally uniformly bounded $C^{1,\a}$ norm and thus the convergence $N_{\k_i}\to N$ is actually a $C^{1,\a}$ convergence, for any $\a\in (0,1)$. This is the result of a standard application of Allard's regularity theorem \cite{Al} on rescalings of $N_{\k_i}$ (see  \cite{DS93, MaMi, Alex}). The uniform $C^{1,\a}$ estimates and standard PDE theory (since the mean curvature of $N_{\k_i}$ is expressed in terms of $|\nabla u_{\hat \e,\k_i}|$, see \cite{GT01}), imply now that we have locally uniform $C^{\infty}$ estimates for the graphs $N_{\k_i}$ and, as a consequence, the convergence $N_{\k_i}\to N$ is smooth.

We now claim that, as a consequence of the Hopf maximum principle,   the components of the limit $N$ are  embedded graphs. To see this, we rework the Jacobi equation (\ref{jacobi}) to express it instead in terms of the vertical component $\frac{1}{v}$ of the upper unit normal vector $\nu$ to $N_{\kappa}$, which yields
\begin{equation*}
\Delta^{N_k}\left(\frac{1}{v}\right)+\frac{1}{v}\left(|A|^2+\left(1-\frac{1}{v^2}\right)\Ric(\gamma,\gamma)-v\bar g\left(\nabla ^{N_k}H,\tau\right)\right)=0.
\end{equation*}
Then, using the equation  $({*_{\hat{\varepsilon},\kappa})}$ to write $H=P-\frac{1}{\e v}+\k\hat u_{\e, \k}$, where $P=(g^{ij}-\nu^i\nu^j)K_{ij}$, along with the estimate
\begin{equation*}
|A|^2+\left(1-\frac{1}{v^2}\right)\Ric(\gamma,\gamma)-v\bar g\left(\nabla^{N_k}(\kappa \hat{u}_{\e,\k}+P),\tau\right)\geq \frac{\kappa |\nabla^{N_k}\hat{u}_{\e,\k}|^2}{v}-\beta\geq-\beta,
\end{equation*}
for some constant $\beta\geq0$ depending on the size of the Ricci tensor and $\|K\|_{C^1}$ (see \eqref{MNest}), we see that the vertical component of the graph satisfies
\begin{equation}\label{Harnack}
\Delta^{N_k}\left(\frac{1}{v}\right)+\frac{1}{\varepsilon}\bar g\left(\nabla^{N_k} \left(\frac{1}{v}\right),\tau\right)\leq\frac{\beta}{v}.
\end{equation}
The fact that the supremum and gradient estimates for $\hat{u}_{\varepsilon,\kappa}$ (Lemmas \ref{supest} and \ref{interiorforkappa}) blow up as $\kappa\to0$, together with equation $(\ref{Harnack})$, then leads to the  following classification of the components of the limit $N$ of $N_{\kappa}$. This blowup analysis follows as in  \cite[Proposition 4]{SY81} (see also \cite{E09}). 
\begin{Theorem}\label{kto0}
Assume that $2\le n\le 6$ and let $(M^{n+1},g,K)$ be an initial data set and let $\partial\Omega$ be a smooth, strictly mean convex and outer untrapped  hypersurface in $M$. Then, for  $\e\le \min\left\{\frac{1}{(n+1)\lambda},\frac12\right\}$, where $\lambda=\max_i\{|\l_i|, \l_i\text{  eigenvalue of  } K\}$, there exists a sequence $\{\k_i\}_{i\in\N}$ with $\kappa_i\downarrow0$, together with an open and connected set $\Omega_\e$  such that if $\hat{u}_{\varepsilon,\k_i}$ solves $(*_{\hat{\varepsilon},\k_i})$ the following hold.

$(i)$ The sequence $\{\hat{u}_{\varepsilon, \k_i}\}_{i\in\N}$ converges  uniformly to $+\infty$ on $\partial\Omega_\e\setminus \partial \Omega$, and 
$\hat{u}_{\varepsilon,\k_i}$ converges locally smoothly to $\hat{u}_{\varepsilon}$ in $\Omega_\e$, where $\hat{u}_{\varepsilon}$ is a smooth function  that satisfies $(*_{\hat{\varepsilon}})$ in ${\Omega}_\e$.

$(ii)$  Each boundary component $\Sigma_{\e}$ of $\partial \Omega_\e\setminus\partial\Omega$ is an embedded MOTS satisfying $\theta^+_{\Sigma_\e}=H_{\Sigma_{\e}}+{\tr}_{\Sigma_{\e}}K=0$, where $H_{\Sigma_{\e}}$ is the mean curvature of $\Sigma_\e$ taken with respect to the inward pointing unit  normal to $\Omega_\e$.

$(iii)$ If $\Omega$ does not contain a closed MOTS in its interior, $\hat{u}_{\varepsilon,\k_i}$ converges to a smooth solution $\hat{u}_{\varepsilon}$ of $(*_{\hat{\varepsilon}})$ defined on all of $\bar{\Omega}$.  
\end{Theorem}
\begin{proof}
As we explained before the statement  of the theorem, using standard results of $(C,1)$-minimal sets (see \cite{MaMi}), we have  that the graphs $N_{\k_ i}$ of the functions $\hat{u}_{\varepsilon,\k_i}$ converge locally smoothly to  a smooth embedded hypersurface $N$ in $\bar{\Omega}\times\mathbb{R}$. Moreover, since $N$ inherits its orientation from $N_{\k_i}$,  it follows that $\frac{1}{v_{\k_i}}=(\sqrt{1+|\nabla\hat{u}_{\varepsilon, \k_i}|^2})^{-1}$ converges (smoothly) to the vertical component, $\frac{1}{v}$, of the unit normal vector of $N$.   In view of \eqref{Harnack}, this limit satisfies $\Delta^N\left(\frac{1}{v}\right)+\frac{1}{\varepsilon}\bar g\left(\nabla^N \frac{1}{v},\tau\right)-\frac{\beta}{v}\leq0$. The Hopf maximum principle then says that on each connected component of $N$, we have that $\frac{1}{v}$ either vanishes identically---and the connected component is cylindrical---or else is everywhere positive---and the connected component is a graph. Note that here no component can be a cylinder, since the functions $\hat{u}_{\varepsilon,\k_i}$ are non-negative. Furthermore, the boundary gradient estimates given in Lemma~\ref{boundaryforkappa} ensures that the graphs $N_{\k_i}$ must remain bounded near $\partial\Omega$, and thus  the limit $N$ is a graph near $\partial\Omega$.
Therefore, $N$ is the graph of a function, which we call $\hat{u}_{\varepsilon}$, defined on an open (non-empty) subset of $\Omega$, which we call $\Omega_\e$. The locally smooth convergence $\hat u_{\e, \k_i}\to\hat u_\e$ then immediately yields that $\hat{u}_{\varepsilon}$ satisfies $(*_{\hat\varepsilon})$ in ${\Omega}_\e$ and diverges to infinity on approach to $\partial\Omega_\e\setminus\partial\Omega$.  This finishes the proof of (i).

To prove (ii), we need to show that the set $\Sigma_\e=\partial\Omega_\e\setminus \partial \Omega$, where the function $\hat u_\e$ tends to infinity, is a MOTS (note $\Sigma_\e$ as defined here might have more than one connected component). Since $N=\graph \hat u_\e$ over $\Omega_\e$ and $\hat u_\e$ satisfies $(*_{\hat\varepsilon})$, we have that
$H+P=\frac{1}{\e\sqrt{1+|\nabla u_{\hat\e}|^2}}=\frac 1\e \nu^{n+2}$, where $H=H(x)$ is the mean curvature of $N$ at $(x, \hat u_\e(x))$, $P=P(\nu)=(g^{ij}-\nu^i\nu^j)K_{ij}$, $\nu=\nu(x)$ is the upward pointing unit normal to $N$ at $(x, \hat u_\e(x))$ and $\nu^{n+2}$ is its vertical component. We consider vertical translations, $N_{\a_i}= N-\a_i$, of $N$ for a sequence $\{\a_i\}_{i\in\N}\subset \R$ with $a_i\uparrow \infty$. $N_{\a_i}$ have uniformly bounded mean curvature and thus are  $(C,1)$-minimal sets. Therefore, we can argue as with the convergence $N_{\k_ i}\to N$, using the results of \cite{DS93, MaMi}, to conclude that, after passing to a subsequence, $N_{\a_i}\to \wt N$ locally smoothly (note again that the mean curvature of $N_{\a_i}$ can be expressed in terms of its normal). Since we also have that $N_{\a_i}\to \Sigma_\e\times \R$ locally uniformly, we conclude that $\wt N =\Sigma_\e\times \R$. The locally smooth convergence $N_{\a_i}\to \Sigma_\e\times \R$, along with the fact that for $N=\graph \hat u_\e$ we have
$H+P=\frac 1\e \nu^{n+2}$, implies that the limit  $\Sigma_\e\times \R$ is a MOTS.

Finally, we note that if $\Omega$ does not contain a closed MOTS in its interior, then $\Sigma_\e=\emptyset$ and therefore (iii) holds.
\end{proof}

\begin{Remark} It is interesting to observe that the elliptic regularization problem $(*_{\hat{\varepsilon}})$ provides a new way to locate MOTS in space-time initial data sets with a mean convex, outer-untrapped hypersurface. In the following section we will show that the hypersufaces $\Sigma_\e$ are not only  MOTS but they are indeed  the outermost MOTS and thus this is actually a way to  locate the outermost MOTS.
\end{Remark}

\section{Convergence to the outermost MOTS}\label{outermostMOTSsection}

In this section we will show that the set where the functions $\hat u_\e$ blow up---that is, the inner boundary of the set $\Omega_\e$ as defined in Theorem~\ref{kto0}---is not only a MOTS but it is actually the outermost MOTS. We will do this by modifying the initial data $K$ inside the outermost MOTS.

There is a notion of stability for MOTS analogous to the notion of stability for minimal hypersurfaces (see \cite{AMS08}) which allows for many results  from the case of stable minimal hypersurfaces to be generalized in the case of stable MOTS, even though the stability operator in the case of MOTS is not
self-adjoint. 
It is known that the outermost MOTS, $\Sigma_\out=\partial \Omega_\out$, is stable (see \cite{AM09}), something that  was used in  the proof of \cite[Theorem 5.1]{AM09} to show that one can change the  initial data $K$ in $\Omega_\out$, so that there exists a smooth outer trapped hypersurface $\Sigma^-$ (i.e. satisfying $\theta^+(\Sigma^-)=H+P<0$) inside $\Sigma_\out$ (i.e. $\Sigma^-\subset\Omega_\out$). 

In order to prove that the functions $\hat u_\e$ (as defined in Theorem  \ref{kto0}) blow up over the outermost MOTS, we show that they satisfy $\hat u_\e\ge \delta^{-1}$ over  $\Sigma^-$ (with $\Sigma^-$ as above) for any constant $\delta>0$. To do this we will flow $\Sigma^-$ by  smooth null mean curvature flow, as defined in $(*)$, in order to create lower barriers for the solutions $\hat u_{\e}$ of the equations $({*_{\hat{\varepsilon}}})$ which are greater than $\delta^{-1}$ over (and inside) $\Sigma^-$.
Before we make this rigorous, we recall the construction of $\Sigma^-$ in \cite[Theorem 5.1]{AM09} as we would like to make some minor modifications.
Let $\psi>0$ be the principal eigenfunction of the stability operator (which is derived by the variation of $\theta^+$, see \cite{AMS08, AM09}) and extend the vector field $\psi\nu$ to a neighborhood of $\Sigma_\out$, where $\nu$ is the outward pointing unit normal to $\Sigma_\out$. By flowing $\Sigma_\out$ in the direction $-\psi\nu$, we construct, for some $\s>0$, a foliation $\{\Sigma_{\out,t}\}_{t\in(-4\s,0]}$ of a neighborhood of $\Sigma_\out$, such that $\Sigma_{\out, 0}=\Sigma_\out$, $\Sigma_{\out, t}$  lies inside  $\Sigma_\out$ (i.e. $\Sigma_{\out, t}\subset\Omega_\out$) for all $t\in (-4\s,0)$ and 
\begin{equation}\label{KK'}
\frac{\partial}{\partial t}\bigg|_{t=0}\theta^+(\Sigma_{\out, t})=0.
\end{equation}
We define then the new data by
\begin{equation}\label{K'}
K'=K-\frac{1}{n}\phi(t)g,
\end{equation}
where  $\phi:\R\to\R$ will be chosen momentarily. Then, with respect to the new data, the null mean curvature of the hypersurfaces $\Sigma_{\out, t}$, $\theta_{K'}^+(\Sigma_{\out, t})$, is given by
\[\begin{split}\theta_{K'}^+(\Sigma_{\out, t})&=(H+P)(\Sigma_{\out,t})=\dvg\nu_t+(g^{ij}-\nu_t^i\nu_t^j)K'_{ij}\\
&=\dvg\nu_t+(g^{ij}-\nu_t^i\nu_t^j)\left(K_{ij}-\frac{1}{n}\phi(t)g_{ij}\right)=\theta^+(\Sigma_{\out, t})-\phi(t).
\end{split}\]
We now choose $\phi$ to be such that $\phi(t)=0$ for  $t> 0$ so that $K'=K$ outside $\Sigma_\out$.
Moreover, since $\theta^+(\Sigma_{\out, t})$ vanishes to first order in $t$ at $t = 0$ by \eqref{KK'},  $\phi$ can be chosen so that it is $C^{1,1}$,  $\theta_{K'}^+(\Sigma_{\out, t})<0$ for all $t\in(-4\s,0)$  and  $\|K'\|_{C^1}\le 2\|K\|_{C^1}$. 
In fact, we can also choose $\phi$ so that  the eigenvalues of $K'$ are controlled in the region foliated by $\{\Sigma_{\out,t}\}_{t\in(-4\s,-2\s]}$, by paying with the fact that $\|K'\|_{C^1}$ will now depend not only on $\|K\|_{C^1}$, but also on $\s$: Setting $U_{2\s}=\{\Sigma_{\out,t}\}_{t\in(-4\s,-2\s]}$, we choose $\phi$ so that the new data have the additional property that
for any $v\in \R^{n+1}$
\begin{equation}\label{lK}
 v^iv^jK'_{ij}=v^iv^jK_{ij}-\frac{\phi}{n} |v|^2\le\left(\l_{\max}-\frac{\phi}{n}\right)|v|^2\le0\text{   in   }U_{2\s},
 \end{equation}
where  $\lambda_{\max}=\max_i\{\l_i, \l_i\text{  eigenvalue of  } K\}$. The new data now satisfies $\|K'\|_{C^1}\le C(\|K\|_{C^1},\s)$ and $K'=K$ outside $\Sigma_\out$.

Henceforth we take $\Sigma^-$ to be one of the leaves $\Sigma_{\out, t}$ for some $t\in(-4\s,-3\s)$   so that 
\begin{equation}\label{P-}
(H+P)(\Sigma^-)=\dvg\nu+( g^{ij}-\nu^i\nu^j)K'_{ij}<0,
\end{equation}
where  $\nu$ denotes the outward  pointing unit normal to $\Sigma^-$. 

\begin{Remark}\label{SUrmk} Let $\Sigma^-$ be as above,  satisfying \eqref{P-}. We then have short time existence of a (smooth) solution of the equation
\begin{equation}\label{eqnP-}
\left\lbrace
\begin{aligned}
 \frac{\partial F^-}{\partial t}(x,t) & = -(H+P)(x,t)\nu(x,t), \quad x\in\Sigma^-,\,t\geq0, \\
                         F^-(\Sigma^-,0) & = \Sigma^-,
\end{aligned}
\right.
\end{equation}
where $H+P$ is defined using the new data $K'$ as in  \eqref{K'} and satisfying also \eqref{lK} (see for example \cite{CB, GP99}). That is, there exists $T>0$ and a smooth solution  $F^-:\Sigma^-\times[0,T)\to M$
of  \eqref{eqnP-}. The evolution equation of $H+ P$ (given  in \eqref{heH+P}) along
with  the maximum principle and \eqref{P-}, implies that $\Sigma^-$ flows towards $\Sigma_\out$ and
 $F^-(\Sigma^-, {t_1})\cap F^-(\Sigma^-, {t_2})=\emptyset$ for any $t_1\ne t_2$. For $\tau=\min\{T/2,\s\}$, we let 
 \[ U_\t:=\bigcup_{0<t<\tau}F^-(\Sigma^-, {t})\subset M\]
and let $u^-:U_\t\to \R$ be defined by $u^-(p)=t\Leftrightarrow p\in F^-(\Sigma^-, {t})$. Then, we have that $u^-$ is a smooth solution of the following equation over $\ov U_\t$ 
\begin{equation}\label{ellH+Ps}
\dvg\left( \frac{\nabla u^-}{|\nabla u^-|}\right) +\left(g^{ij}-\dfrac{\nabla^iu^-\nabla^ju^-}{|\nabla u^-|^2}\right)K'_{ij}= \frac{-1}{|\nabla u^-|}
\end{equation}
(cf. $(**)$ and note the change of sign in front of the $K'$-term on the left-hand side), such that $u^-=0$ on $\Sigma^-$ and $u^-=\tau$ on $\partial U_\t\setminus\Sigma^-=F^-(\Sigma^-, \tau)$. Furthermore, there exists some constant $C_0\ge1$ such that
\begin{equation}\label{Dus}
\frac{1}{C_0}\le |\nabla u^-|\le C_0\,\,,\,\,|\nabla^2 u^-|\le C_0\text{    in   }\ov U_\t.
\end{equation}

\end{Remark}

We will show that, for an appropriately chosen $\psi$ and with $u^-$ as in Remark~\ref{SUrmk}, the function $\psi\circ u^-$ is a lower barrier for solutions $\hat u_{\e}$ of the equation (${*_{\hat{\varepsilon}}}$).
The idea of {\it bending} the (short time) smooth solution to get boundary barriers for the approximating solutions is applied in  \cite{S08}, where in \cite[Lemma 4.2]{S08} such a construction was used for the mean curvature flow.

\begin{Lemma}\label{barrier}
Let $\Sigma^-$, $\tau$, $U_\t$, $u^-$ and $C_0$ be as in Remark~\ref{SUrmk}. Let $U\subset M$ be such that $U_\t\subset U$ and $\partial U=\partial U_\t\setminus\Sigma^-=F^-(\Sigma^-, \tau)$, and extend $u^-$ in $U$ so that it is zero in $U\setminus U_\t$. Then, for any $\delta>0$ there exists a $C^2$ function $\psi:\R\to\R$ such that the following holds.  The function $v=\psi\circ u^-:U\to\R$ is a $C^2$ function such that $\forall \e\le\e_1=\e_1(C_0,\t)$  (a constant that depends only on $C_0$ and $\t$)
\[\begin{split}\mathcal M_\e(v):=& \left(g^{ij}-\frac{\nabla^iv\nabla^jv}{|\nabla v|^2+1}\right)\nabla^{ij}v\\
&-\left(g^{ij}-\dfrac{\nabla^iv \nabla^j v}{|\nabla v|^2+1}\right)K'_{ij} \sqrt{|\nabla v|^2+1}+\frac{1}{\e}\ge 0
\end{split}\]
 and 
\[\begin{split}v= 0&\text{ on  }\partial U, \\
v\ge \frac{1}{\delta}&\text{ in  }U\setminus U_\t.
\end{split}\]
\end{Lemma}
\begin{proof}
Omitting the  ``$-$'' superscript for simplicity, thus writing $u=u^-$, we set $v=\psi(u)$, where $\psi:[0,\t]\to\R$ is a $C^2$ function. Provided that $v\in C^2(U)$, we then have
\begin{equation}\label{MvMueqn}
\begin{split}
\mathcal M(v):=&\frac{1}{|\nabla v|}\left(g^{ij}-\frac{\nabla^i v\nabla^j v}{|\nabla v|^2}\right)\nabla_{ij}v+\left(g^{ij}- \frac{\nabla^i v\nabla^j v}{|\nabla v|^2}\right) K_{ij} +\frac{1}{|\nabla v|}\\
=&\frac{1}{{\psi'|\nabla u|}}\left({g^{ij}}-\frac{(\psi')^2\nabla^i u\nabla^j u}{(\psi')^2|\nabla u|^2}\right)(\psi'\nabla_{ij}u+\psi''\nabla_i u\nabla_ju)\\
&+\left(g^{ij}- \frac{(\psi')^2\nabla^i u\nabla^j u}{(\psi')^2|\nabla u|^2}\right) K_{ij} +\frac{1}{|\nabla v|}\\
=&\mathcal M(u)-\frac{1}{|\nabla u|}+\frac{1}{|\nabla v|}=\frac{1-\psi'}{|\nabla v|},
\end{split}
\end{equation}
where we have used \eqref{ellH+Ps} which implies that $\mathcal M(u)=0$ (note that in the above calculation, the terms involving $\psi''$ cancel).
 We now compute  
\begin{equation}\label{Mk}\begin{split}\mathcal M_\e(v)=&|\nabla v| \mathcal M(v) -1+\frac{1}{\e}+\left(\frac{\nabla^iv\nabla^jv}{|\nabla v|^2}-\frac{\nabla^iv\nabla^jv}{|\nabla v|^2+1}\right)\nabla_{ij}v\\
&-|\nabla v|\left(g^{ij}-\dfrac{\nabla^iv \nabla^j v}{|\nabla v|^2}\right)K'_{ij}-\left(g^{ij}-\dfrac{\nabla^iv \nabla^j v}{|\nabla v|^2+1}\right)K'_{ij} \sqrt{|\nabla v|^2+1}\\
=&-\psi'+\frac1\e +\frac{\nabla^i u\nabla^j u}{|\nabla u|^2(1+|\nabla v|^2)}(\psi'\nabla_{ij}u+\psi''\nabla_i u\nabla _ju)\\
&-|\nabla v|\left(g^{ij}-\dfrac{\nabla^iv \nabla^j v}{|\nabla v|^2}\right)K'_{ij}-\left(g^{ij}-\dfrac{\nabla^iv \nabla^j v}{|\nabla v|^2+1}\right)K'_{ij} \sqrt{|\nabla v|^2+1}.
\end{split}
\end{equation}
We estimate the terms on the right-hand side of \eqref{Mk}, using the estimates \eqref{Dus} and the property of $K'$  \eqref{lK}, as follows.
\begin{equation}\label{term1}
\begin{split}
\frac{\nabla^i u\nabla^j u}{|\nabla u|^2(1+|\nabla v|^2)}(\psi'\nabla_{ij}u+\psi''\nabla_i u\nabla _ju)&\ge -\frac{|\psi''||\nabla u|^2+C_0|\psi'|}{1+(\psi')^2|\nabla u|^2}\\
&\ge -\left(C_0+\min\left\{C_0^2|\psi''|,\frac{|\psi''|}{(\psi')^2}\right\}\right),
\end{split}
\end{equation}
and
\begin{equation}\label{term2}
\begin{split}
-|\nabla v|\left(g^{ij}-\dfrac{\nabla^iv \nabla^j v}{|\nabla v|^2}\right)K'_{ij}-&\left(g^{ij}-\dfrac{\nabla^iv \nabla^j v}{|\nabla v|^2+1}\right)K'_{ij} \sqrt{|\nabla v|^2+1}\\
\ge&(|\nabla v|+\sqrt{|\nabla v|^2+1})\left(-\tr K'+\l'_{\min}\right)\ge 0,
\end{split}
\end{equation}
where $\l'_{\min}=\min_i\{\l'_i, \l'_i\text{  eigenvalue of  } K'\}$. We claim now  that there exists a $C^2$ function $\psi:[0,\tau]\to \R$ such that $v=\psi\circ u\in C^2(U)$ and such that
\[v(x)=\psi(u(x))=\begin{cases}\psi(\tau)=0&\text{ on  }\partial U,\\
\psi(0)\ge \frac{1}{\delta} &\text{ in  } U\setminus U_\t,
\end{cases}\]
and
\begin{equation}\label{psi'}
-\psi'-\min\left\{C_0^2|\psi''|,\frac{|\psi''|}{(\psi')^2}\right\}\ge -C=-C(\tau, C_0).
\end{equation}
The existence of such a function $\psi$ implies then, after using the estimates \eqref{term1}, \eqref{term2} and \eqref{psi'} in \eqref{Mk}, 
\[
\mathcal M_\e(v)\ge\frac{1}{\e}-C_0-C,\]
where $C=C(\tau, C_0)$ is the constant from the estimate  \eqref{psi'}, and thus taking $\e_1=\e_1(C_0,\t)=({C_0+C})^{-1}$ we have that for all $\e\le \e_1$
\[\mathcal M_\e(v)\ge 0.\]
This concludes the proof of the lemma, provided that there exists a function $\psi$ as we claimed above and which we now construct.

For any $\d\in(0,1]$, we define $\zeta:[0,2]\to\R$ by
\[\zeta(t)=\begin{cases}\log\left(\frac{t}{\d}+1\right)&\text{ for  } t\in\left[0,1\right],\\
c_1+c_0\left(t-1\right)-\frac{c_0^2}{2}\left(t-1\right)^2\\
\hspace{.4cm}+\frac{2c_0^2-3c_0}{3}\left(t-1\right)^3+\frac{-c_0^2+2c_0}{4}\left(t-1\right)^4 &\text{ for  } t\in\left[1,2\right],
\end{cases}\]
where $c_1=\log\left(\frac1\d+1\right)$ and $c_0=\frac{1}{1+\d}$. Note first that
\[\zeta(t)=\begin{cases}0&\text{ for  } t=0,\\
c_1+\frac{c_0}{2}-\frac{c_0^2}{12}\ge c_1=\log\left(\frac1\d+1\right)&\text{ for  } t=2.
\end{cases}\]
For the derivatives of $\zeta$, we have
\[\zeta'(t)=\begin{cases}\frac{1}{t+\d}&\text{ for  } t\in\left(0,1\right],\\
\frac{1}{1+\d}=c_0&\text{ for  } t=1,\\
c_0-{c_0^2}\left(t-1\right)+({2c_0^2-3c_0})\left(t-1\right)^2,\\
\hspace{.4cm}+({-c_0^2+2c_0})\left(t-1\right)^3 &\text{ for  } t\in\left[1,2\right]\\
0&\text{ for  } t=2,
\end{cases}\]
and
\[\zeta''(t)=\begin{cases}-\frac{1}{(t+\d)^2}&\text{ for  } t\in\left(0,1\right],\\
-\frac{1}{(1+\d)^2}=-c_0^2&\text{ for  } t=1,\\
-{c_0^2}+({4c_0^2-6c_0})\left(t-1\right)+({-3c_0^2+6c_0})\left(t-1\right)^2 &\text{ for  } t\in\left[1,2\right],\\
0&\text{ for  } t=2,
\end{cases}\]
where of course here at $t=2$ we mean the  left derivatives, and hence $\zeta$ is a $C^2$ function. We further note that
\[\zeta'-\min\left\{|\zeta''|,\frac{|\zeta''|}{(\zeta')^2}\right\}\ge\begin{cases}\zeta' -\frac{|\zeta''|}{(\zeta')^2}\ge -1&\text{ for  } t\in\left(0,1\right],\\
 \zeta'-|\zeta''|\ge -20c_0\ge -20&\text{ for  } t\in\left[1,2\right].
 \end{cases}
\]
We now define the function $\psi:[0,\tau]\to\R$ by 
\[\psi(t)=\zeta\left(\frac{2}{\tau}(\tau-t)\right)\]
and claim that this is the desired function. Note first that 
 $\psi\in C^2((0,\tau))$  and for the (right) derivatives at zero we have
\[\psi'(0)=-\frac{2}{\t}\zeta'(2)=0\,\,,\,\,\psi''(0)=\frac{4}{\t^2}\zeta''(2)=0.\]
Hence, the function $v=\psi\circ u$ is also $C^2$ and satisfies
\[v(x)=\psi(u(x))=\begin{cases}\psi(\tau)=\zeta(0)=0&\text{ on  }\partial U,\\
\psi(0)=\zeta(2)\ge c_1=\log\left(\frac1\d+1\right) &\text{ in  } U\setminus U_\t.
\end{cases}\]
Finally, we have
\[\begin{split}-\psi'(t)&-\min\left\{C_0^2|\psi''|,\frac{|\psi''|}{(\psi')^2}\right\}\\
&=\frac{2}{\tau}\zeta'\left(\frac2\tau(\tau-t)\right)-\min\left\{\frac{4 C_0^2}{\tau ^2}|\zeta''|,\frac{|\zeta''|}{(\zeta')^2}\right\}
\ge -20\frac{4 C_0^2}{\tau^2}.
\end{split}\]
Therefore, the function $\psi$ as defined above has all the required properties, after replacing $\d$ by $(e^{1/\d}-1)^{-1}$.
\end{proof}

As a direct consequence of Lemma~\ref{barrier} and the comparison principle we obtain the following. 
\begin{Theorem}\label{blowup}
There exists an $\e_0$ depending only on the initial data, such that for any $\e\le \e_0$ the following holds. There exists a solution $\hat u_\e\in C^\infty(\Omega_0)$ of the equation ($*_{\hat\e}$), where  $\Omega_0=\Omega\setminus\ov\Omega_\infty$, such that $\hat u_\e$ blows up over the inner boundary $\partial\Omega_0\setminus\partial\Omega=\Sigma_\out$ (the outermost MOTS) and is  zero over the outer boundary $\partial \Omega$.

 Consequently the function $u_\e=\e\hat u_\e$ is then a smooth solution of ($*_{\e}$) in $\Omega_0$ that blows up over the inner boundary $\Sigma_\out$ (the outermost MOTS) and is  zero over the outer boundary $\partial \Omega$.
\end{Theorem}
\begin{Remark} The $\e_0$ of the theorem is given by $\e_0=\min\left\{\frac12, \frac{1}{(n+1)\l'}, \e_1\right\}$, where $\l'=\max_i\{|\l'_i|, \l_i\text{  eigenvalue of  } K'\}$ and $\e_1=\e_1(C_0,\t)$, $C_0$ and $\t$ are as in  Lemma~\ref{barrier}. The reason for the dependence of $\e_0$ on $K'$, instead of $K$, is that we want Theorem~\ref{kto0} to hold with $K$ replaced by the new data $K'$. 
\end{Remark}

\begin{proof}[Proof of Theorem~\ref{blowup}] 
We will make use of  the new data $K'$ as in  \eqref{K'}, which also satisfy \eqref{lK}. Let ($*_{\hat \e, \k}$)' and ($*_{\hat\e}$)' denote the equations ($*_{\hat \e, \k}$) and ($*_{\hat\e}$) after we have replaced $K$ by $K'$.
Note first that we can repeat the estimates of Sections \ref{ellregsection} and \ref{eexistencesection} with the new data $K'$ (in the place of $K$) and thus Theorem~\ref{kto0} holds with $K$ replaced by $K'$ and equations ($*_{\hat \e, \k}$) and ($*_{\hat\e}$) replaced by ($*_{\hat \e, \k}$)' and ($*_{\hat\e}$)'. Theorem~\ref{kto0} (i) then implies that for any $\e\le \min\left\{\frac12, \frac{1}{(n+1)\l'}\right\}$, where $\lambda'=\max_i\{|\l'_i|, \l_i\text{  eigenvalue of  } K'\}$, there exists an open and connected set $\Omega_\e\subset \Omega$ and a solution  $\hat u_\e\in C^\infty(\Omega_\e)$ of ($*_{\hat\e}$)', with $\hat u_\e=0$ on $\partial \Omega\subset \partial\Omega_\e$ and $\hat u_\e$ blowing up on the other boundary components. Furthermore, by  Theorem~\ref{kto0} (ii), we have that $\partial \Omega_\e\setminus\partial \Omega$ is a MOTS. 

Let now $\Sigma^-$,  $U$,  and $v$ be as in Lemma~\ref{barrier} for some $\delta>0$ . Then, for $\e\le \min\left\{\frac12, \frac{1}{(n+1)\l}, \e_1\right\}$ (where  $\e_1$ is as in  Lemma~\ref{barrier}), by the comparison principle we have that $v\le \hat u_\e$ over $U$ and thus $\frac{1}{\delta}\le \hat u_\e$ over $\Sigma^-$. Since this is true for any $\delta>0$ we obtain  that the MOTS $\partial\Omega_\e\setminus \Omega$ must lie outside $\Sigma^-$. Recall now that by the construction of the new data $K'$ (in the beginning of this section) the region between $\Sigma^-$ and the outermost MOTS, $\Sigma_\out$, is foliated by outer trapped hypersurfaces $\Sigma_{out,t}$. The maximum principle then implies that the MOTS $\partial\Omega_\e\setminus \Omega$ cannot enter the open region between $\Sigma^-$ and the outermost MOTS and therefore it must coincide with the outermost MOTS, $\Sigma_\out$.
\end{proof}

\section{The limit of solutions to $(*_{\varepsilon})$}
In Sections \ref{eexistencesection} and \ref{outermostMOTSsection} we established existence of solutions $u_{\varepsilon}=\varepsilon\hat{u}_{\varepsilon}$ to the null mean curvature flow elliptic regularization problem $(*_{\varepsilon})$ in ${\Omega}_0\subset {\Omega}$ for $\varepsilon\le \e_0$ (a constant that depends only on the initial data), where $\Omega_0$ is as in Theorem~\ref{blowup}, so that $\partial \Omega_0\setminus\partial \Omega=\Sigma_\infty$,  the outermost MOTS.
We want to send $\varepsilon\to0$ to obtain a weak solution to $(**)$. However, the interior and boundary gradient estimates for $(*_{\hat{\varepsilon},\kappa,s})$ derived in Lemmas \ref{interiorforkappa} and \ref{boundaryforkappa} both rely on the supremum estimate for $\hat{u}_{\varepsilon,\kappa}$. Since the supremum bound of Lemma~\ref{supest} blows up when we take the limit $\kappa\to0$, these a-priori estimates do not hold in the limit $\kappa\to0$, and thus they are of no use in extracting the limit for $\varepsilon\to0$ of the solution $u_{\varepsilon}$ to $(*_{\varepsilon})$. Therefore, we must derive new interior and boundary gradient estimates for $(*_{\varepsilon})$ that are uniform in $\varepsilon$.

\begin{Lemma}[Uniform Gradient Estimate]\label{unigradest} Let $\e\le \frac12$ and $u_\e\in C^\infty(\Omega_0)$ be a solution  of $(*_{\varepsilon})$  as  in  Theorem~\ref{blowup}. Then, $u_\e$  satisfies the estimate
\begin{equation*}
\sup_{\Omega_{T/2}}|\nabla u_{\varepsilon}|\leq \frac{2}{\eta T}\exp(\eta T)\cdot\sup_{\partial\Omega}(1+\sqrt{\varepsilon^2+|\nabla u_{\varepsilon}^2|}),
\end{equation*}
where $\eta$ is a constant that depends only on the initial data, in fact $\eta=\eta(n,\Ric,\|K\|_{C^1})$, and $\Omega_T=\{x\in\bar{\Omega}_0:u_{\varepsilon}(x)\leq T\}$.
\end{Lemma}
\begin{proof} We take a similar approach to that of the proof of Lemma~\ref{interiorforkappa}. Let $N=\graph \hat u_\e$ and $v=\sqrt{1+|\nabla \hat u_\e|^2}$, where recall that $\hat u_\e=\e u_\e$ is a solution of $(*_{\hat{\varepsilon}})$. Let also $w(x,z):=\exp(-\varepsilon\eta z)$ for $(x,z)\in M\times\R$ and $w_0=\exp(-\eta T)$, so that $w-w_0=0$ when $z=\e^{-1}T$. We compute $\Delta^N((w-w_0)v)$ on $N$, similarly to \eqref{jacobi2} in the proof of Lemma  \ref{interiorforkappa}, as follows. We first note that
\[
\nabla^N w=-\e\eta w\left(\tau-\frac1v\nu\right)\,,\,\,\,\Delta^N w=\e^2\eta^2\left(1-\frac{1}{v^2}\right)w+\e\eta\frac{H}{v}w,
\]
where the notation here and throughout this proof is as in the proof of Lemma  \ref{interiorforkappa}. Thus (using \eqref{jacobi} from the proof of Lemma  \ref{interiorforkappa}) we obtain
\begin{equation}\label{newJacobiEqn}
\begin{split}
\Delta^N((w-w_0)v)=&\frac{2}{v}\bar g(\nabla^Nv,\nabla^N((w-w_0)v))
\\
&+(w-w_0)v\Bigg(|A|^2+\left(1-\frac{1}{v^2}\right)\Ric(\gamma,\gamma)\\
&+(\varepsilon\eta)^2\left(1-\frac{1}{v^2}\right)+\varepsilon\eta\frac{H}{v}-v\bar g(\nabla^NH,\tau)\Bigg)\\
&+w_0v\left((\varepsilon\eta)^2\left(1-\frac{1}{v^2}\right)+\varepsilon\eta\frac{H}{v}\right).
\end{split}
\end{equation}
To argue by contradiction, define $C_1:=\sup_{\partial\Omega}\varepsilon\sqrt{1+|\nabla \hat{u}_\e|^2}$ and assume 
\begin{equation}\label{contradiction}
\sup_{\Omega_T}((\exp(-\eta \varepsilon\hat{u})-\exp(-\eta T))\varepsilon\sqrt{1+|\nabla \hat{u}_\e|^2})>\max\{C_1,1\},
\end{equation}
which must be attained at an interior point $x_0$. Since $N=\graph\hat{u}_\e$, equation $(*_{\hat{\varepsilon}})$ implies that 
\begin{equation}\label{floweqn2}
H+P=\dfrac{1}{\varepsilon v},
\end{equation}
where $H+P$ is the null mean curvature of $N$.
Proceeding as in Lemma~\ref{interiorforkappa}, analogous to (\ref{nablaHest}) (using also \eqref{MNest} and the expression for $|\nabla^Nw|$), we obtain the following estimate
\begin{align*}
(w-w_0)v^2\bar g(\nabla^NH,\tau)=&-\frac{1}{\varepsilon}\bar g(\nabla^N((w-w_0)v),\tau)\\
&-{\eta} wv\left(1-\frac{1}{v^2}\right)-(w-w_0)v^2g(\nabla^NP,\tau)\\
\leq&-\frac{1}{\varepsilon}\bar g(\nabla^N((w-w_0)v),\tau)\\
&-(w-w_0)v\left({\eta}\left(1-\frac{1}{v^2}\right)-2c^2-\frac{|A|^2}{2}\right)\\
&-w_0v{\eta}\left(1-\frac{1}{v^2}\right),
\end{align*}
where $c=c(n, \|K\|_{C^1})$ is the constant from \eqref{MNest}.
At a maximum point $x_0$, where $\Delta((w-w_0)v)\leq0$ and $\nabla((w-w_0)v)=0$, \eqref{newJacobiEqn} reduces to
\begin{align*}
0\geq (w-w_0)v\Bigg(&|A|^2+\left(1-\frac{1}{v^2}\right)\Ric(\gamma,\gamma)\\
&+(\varepsilon\eta)^2\left(1-\frac{1}{v^2}\right)+\varepsilon\eta\frac{H}{v}-v\bar g(\nabla^NH,\tau)\Bigg)\\
&+w_0v\left((\varepsilon\eta)^2\left(1-\frac{1}{v^2}\right)+\varepsilon\eta\frac{H}{v}\right).
\end{align*}
After implementing the above estimates and also using \eqref{floweqn2}, this becomes
\begin{align*}
0\geq(w-w_0)v\Bigg(\frac{|A|^2}{2}&+\left(1-\frac{1}{v^2}\right)\Bigl(\Ric(\gamma,\gamma)+\eta+(\varepsilon\eta)^2\Bigr)\\
&\hspace{3cm}+\frac{\eta}{v^2}-\varepsilon\eta\frac{\|K\|_{C^0}}{v}-2c^2\Bigg)\\
&+w_0v\left((\eta+(\varepsilon\eta)^2)\left(1-\frac{1}{v^2}\right)+\frac{\eta}{v^2}-\varepsilon\eta\frac{\|K\|_{C^0}}{v}\right).
\end{align*}
By the contradiction hypothesis (\ref{contradiction}), we find that $v(x_0)>\frac{1}{\varepsilon}$ and thus $(1-\frac{1}{v^2})>\frac{1}{2}$, provided that $\varepsilon\leq\frac{1}{2}$. Therefore, after discarding some  positive terms from the right-hand side, we obtain
\begin{equation*}
\begin{split}
0&\geq (w-w_0)v\Bigg(\left(1-\frac{1}{v^2}\right)\left({\eta}+\Ric(\gamma,\gamma)-4c^2\right)+\varepsilon^2\eta\left(\frac{\eta}{2}-\|K\|_{C^0}\right)\Bigg)\\
&\quad+w_0\varepsilon^2\eta\left(\frac{\eta}{2}-\|K\|_{C^0}\right),
\end{split}
\end{equation*}
where the constant $c=c(n, \|K\|_{C^1})$ is the constant from \eqref{MNest}. For $\eta=\eta(n,\Ric,\|K\|_{C^1})$ large enough the right-hand side of the above expression becomes  strictly positive, leading to a contradiction. In other words, \eqref{contradiction} cannot be true and therefore we have
\[\sup_{\Omega_T}((\exp(-\eta \varepsilon\hat{u}_\e)-\exp(-\eta T))\varepsilon\sqrt{1+|\nabla \hat{u}_\e|^2})\le\max\{C_1,1\}.\]
For $u_\e=\e\hat u_\e$ we then have
\[\begin{split}\sup_{\Omega_T}((\exp(-\eta{u}_\e)-\exp(-\eta T))\sqrt{\e^2+|\nabla {u}_\e|^2})&\le\max\{C_1,1\}\\
&\le \sup_{\partial\Omega}(1+\sqrt{\e^2+|\nabla {u}_\e|^2}).
\end{split}\]
Restricting now to the region $\Omega_{T/2}$, where 
\[\exp(-\eta{u}_\e)-\exp(-\eta T)\ge \exp(-\eta T/2)-\exp(-\eta T)\ge \frac{\eta T}{2}\exp(-\eta T),\]
we obtain  the required estimate.
\end{proof}

\begin{Lemma}[Uniform boundary gradient estimate]\label{bdrygradest}
There exist constants $C$ and $\e_0$,  depending only on the initial data,  such that for any $\e\le\e_0$ and any solution $u_\e\in C^\infty(\Omega_0)$ of $(*_{\varepsilon})$, as  in  Theorem~\ref{blowup}, the following estimate holds
\[\sup_{\partial\Omega}|\nabla u_{\varepsilon}|\le C.\]
\end{Lemma}
\begin{proof} The idea of the proof is to create an upper barrier for the functions $u_\e$ at the boundary $\partial\Omega$, by bending the (short time) smooth solution of $(*)$ with initial data $\partial \Omega$. This construction is similar to that in \cite[Lemma 4.2]{S08} with the extra complication that here we do not have a supremum estimate for the solutions $u_\e$ (a construction of a barrier using the smooth solution  was also used  in Lemma~\ref{barrier}). 

Let $F(\cdot,t):\partial\Omega\times[0,T)\to M$ be the unique solution to $(*)$, with initial condition $F(\cdot, 0)=\Id_{\partial\Omega\to\partial\Omega}$ and let $\Sigma_t=F(\partial\Omega, t)$ (see Remark~\ref{SUrmk} for the existence of $F$). Since the null mean curvature of the hypersurfaces remains positive (see Remark~\ref{SUrmk}), we obtain that $\Sigma _{t_1}\cap \Sigma_{t_2}=\emptyset$ for $t_1\ne t_2$. For any $\t\in (0, T)$ we define
\[
\Omega_\t=\bigcup_{0<t<\t}\Sigma_t\subset \Omega,
\]
and let $u:\Omega_\t\to \R^+$ be defined by $u(p)=t\Leftrightarrow p\in \Sigma_t$. Then, we have that $u$ is a smooth solution of $(**)$ over $\ov\Omega_\t$ and furthermore there exists some constant $C_0>1$ such that
\begin{equation}\label{Du}
\frac{1}{C_0}\le |\nabla u|\le C_0\,\,,\,\,|\nabla^2 u|\le C_0\text{    in   }\ov \Omega_\t.
\end{equation}

We choose $0<\t<T$ such that $\t<\frac{1}{2}$ and bend the smooth solution $u$ of $(**)$ to construct a supersolution of $(*_{\varepsilon})$ that is zero on $\partial\Omega$ and goes to infinity on the inner boundary $\Sigma_{\tau}$ of $\Omega_{\tau}$. To this end, we define $\psi:[0, \tau)\to \R^+$ to be the following smooth increasing function
\begin{equation}\label{psidelta}
\psi(t)=2t+\frac{1}{\t-t}-\frac{1}{\t}.
\end{equation}
Then $\psi(0)=0$, $\lim_{t\to\t}\psi(t)=+\infty$ and furthermore we have
\begin{equation}\label{Dpsidelta}
\psi'(t)=2+\frac{1}{(\t-t)^2}\,\,,\quad \psi''(t)=\frac{2}{(\t-t)^3}.
\end{equation}
We will show that the function 
\[v(x)=\psi(u(x))\]
is a super solution of $(*_{\varepsilon})$ in $\Omega_\tau$ for sufficiently small  $\e$. Since $u_{\varepsilon}$ solves $(*_{\varepsilon})$ with $u_\e=0$ on $\partial \Omega$, this would then imply that $u_{\varepsilon}\leq v$ on $\bar{\Omega}_{\tau}$ and
\[ \sup_{\partial \Omega}|\nabla u_\e|\le \sup_{\partial \Omega}|\nabla v|\le \left(2+\frac{1}{\t^2}\right) C_0,\]
which proves the lemma with $C=(2+\tau^{-2})C_0$.
Hence, it suffices to show that there exists $\e_0$, depending only on the initial data, such that $v$ is a super solution of $(*_{\varepsilon})$ for all $\e\le \e_0$. 
We first compute, similar to \eqref{MvMueqn} of the proof of Lemma~\ref{barrier},
\[\begin{split}\mathcal M(v):&= \frac{1}{|\nabla v|}\left(g^{ij}-\frac{\nabla^iv\nabla^jv}{|\nabla v|^2}\right)\nabla_{ij}v-\left(g^{ij}-\dfrac{\nabla^iv \nabla^j v}{|\nabla v|^2}\right)K'_{ij}+\frac{1}{|\nabla v|}\\
&=\mathcal M(u)-\frac{1}{|\nabla u|}+\frac{1}{|\nabla v|}=\frac{1-\psi'}{|\nabla v|}.
\end{split}\]
Hence, $v$ is a super solution of $(**)$ if $\psi'\ge 1$. We now relate the level set equation $(**)$ to the elliptic regularized problem $(*_{\varepsilon})$ as follows (cf. \eqref{Mk} in the proof of Lemma~\ref{barrier})
\begin{equation}\label{Mev}
\begin{split}
\mathcal{M}_{\varepsilon}(v):=&\left(g^{ij}-\frac{\nabla^i v\nabla^j v}{|\nabla v|^2+\e^2}\right)\nabla_{ij}v-\left(g^{ij}- \frac{\nabla^i v\nabla^j v}{|\nabla v|^2+\e^2}\right) K_{ij}\sqrt{|\nabla v|^2+\e^2}\\
&+1\\
=&|\nabla v|\mathcal M(v)+ \left(\frac{\nabla^i v\nabla^j v}{|\nabla v|^2}-\frac{\nabla^i v\nabla^j v}{|\nabla v|^2+\e^2}\right)\nabla_{ij}v \\
+&\left(g^{ij}\frac{-\e^2}{|\nabla v|+\sqrt{|\nabla v|^2+\e^2}} +\nabla^iv\nabla^jv\left(\frac{1}{\sqrt{|\nabla v|^2+\e^2}}-\frac{1}{|\nabla v|} \right)\right)K_{ij}
\end{split}\end{equation}
Next, we want to bound the last two terms on the right hand side of \eqref{Mev}. The first of these terms is estimated as follows.
\begin{equation*}\begin{split}
\frac{\e^2\nabla^iv\nabla^jv}{|\nabla v|^2(|\nabla v|^2+\e^2)}\nabla_{ij}v&\le\frac{\e^2}{(\psi')^2|\nabla u|^2+\e^2}(\psi'|\nabla^2u|+\psi''|\nabla u|^2)\\
&\le\e^2\frac{\psi''}{(\psi')^2}+C_0\frac{\e^2\psi'}{(\psi')^2 C_0^{-2}}\le \e^2\frac{\psi''}{(\psi')^2}+C_0^3\frac{\e^2}{\psi'},
\end{split}
\end{equation*}
where we have used (\ref{Du}). Our choice of $\psi$, see (\ref{psidelta}), together with the fact that $\tau<\frac{1}{2}$,  implies
\begin{equation}\label{psi'psi''}
\frac{\psi''}{(\psi')^2}\le2(\t-t)\le2\t\le 1\quad\quad\text{and }\quad\quad\frac{1}{\psi'}\le (\t-t)^2\le1.
\end{equation}
Considering now $\e$ such that $\e\le  C_0^{-2}$, we  obtain the  bound
\[\left(\frac{\nabla^i v\nabla^j v}{|\nabla v|^2}-\frac{\nabla^i v\nabla^j v}{|\nabla v|^2+\e^2}\right)\nabla_{ij}v
\le\frac{1}{C_0^4}+ \frac{1}{C_0}\le \frac{1}{2}.\]
We now bound the second term on the right-hand side of \eqref{Mev} (using again \eqref{psi'psi''})
\[\begin{split}
 &\Biggl(g^{ij}\frac{-\e^2}{|\nabla u|\psi'+\sqrt{(\psi')^2|\nabla u|^2+\e^2}} \\
 &\quad+(\psi')^2\nabla^iu\nabla^ju\left(\frac{1}{\sqrt{(\psi')^2|\nabla u|^2+\e^2}}-\frac{1}{\psi'|\nabla u|} \right)\Biggr)K_{ij}\\
 &\le (n+1)\lambda \left(\e+ |\nabla u|^2(\psi')^2 \frac{\sqrt{(\psi')^2|\nabla u|^2+\e^2}-\psi'|\nabla u|}{\psi'|\nabla u|\sqrt{(\psi')^2|\nabla u|^2+\e^2}}\right)\\
 &= (n+1)\lambda \left(\e+ |\nabla u|\psi' \frac{\e^2}{\sqrt{(\psi')^2|\nabla u|^2+\e^2}(\sqrt{(\psi')^2|\nabla u|^2+\e^2}+\psi'|\nabla u|)}\right)\\
 &\le (n+1)\lambda \left(\e+ \frac{\e^2}{\sqrt{(\psi')^2|\nabla u|^2+\e^2}}\right)\le 2\e(n+1)\lambda \le \frac1 2,
\end{split}\]
with the last inequality being true provided that $\e\le (4(n+1)\lambda)^{-1}$ and  where recall that $\lambda=\max_i\{|\l_i|, \l_i\text{  eigenvalue of  } K\}$. Putting everything together and using these estimates back in \eqref{Mev},  we find that for $\e\le \e_0$, where $\e_0=\min\{(4(n+1)\lambda)^{-1}, C_0^{-2}\}$, we obtain the estimate
\[\mathcal M_\e(v) \le 1-\psi' +{1},\]
which due to \eqref{Dpsidelta} implies
\[
\mathcal M_\e(v) \le 1-\left(2+\frac{1}{(\t-t)^2}\right) +{1}\le-\frac{1}{(\t-t)^2}<0,
\]
so that $v$ is a super solution and thus $u_\e\le v$ for all $\e\le \e_0$.
\end{proof}

We now return to the original elliptic regularization problem $(*_{\varepsilon})$, and note that the a-priori estimates for $u_{\varepsilon}$ given in Lemmas \ref{unigradest} and \ref{bdrygradest}  are uniform in $\varepsilon$. We can therefore use the Arzela--Ascoli theorem to extract a limit as $\e\to 0$. In particular, there exists $u\in C^{0,1}(\Omega_1\cup\partial\Omega)$ and a sequence $\e_k\downarrow 0$ such that
\begin{equation}\label{uetou}
u_{\varepsilon_k}\to u \text{  in $C^0({\Omega_1}\cup\partial\Omega)$},
\end{equation}
 where $\Omega_1\stackrel{open}{\subset}\Omega_0$ is such that $\partial\Omega_1\supset\partial\Omega$. In particular, with $\Omega_2=\cap_t\cup _k\{u_{\e_k}>t\}$, we have $\Omega_1=\Omega_0\setminus\ov\Omega_2$ and thus Lemma~\ref{bdrygradest} implies that $\Omega_1\ne\emptyset$.  Furthermore, since the functions $u_{\e_k}$ tend to $+\infty$ on approach to $\partial\Omega_0\setminus\partial\Omega$, the limit function $u$ also tends to $+\infty$ on approach to  $\partial\Omega_1\setminus\partial\Omega$. With the convergence  `in $C^0({\Omega_1}\cup\partial\Omega)$' above we mean that $u_{\varepsilon_k}\to u$ uniformly in any compact subset of ${\Omega_1}\cup\partial\Omega$. Similarly with `$u\in C^{0,1}(\Omega_1\cup\partial\Omega)$' we mean that $u$ is Lipschitz in any compact subset of ${\Omega_1}\cup\partial\Omega$. Furthermore, Lemma~\ref{unigradest}, along with the Banach--Alaoglu theorem,  implies that 
\begin{equation}\label{Duetou}
\int_{\Omega_1}\nabla u_{\e_k}\cdot fd\H^n\to \int_{\Omega_1}\nabla u\cdot fd\H^n\,,\,\,\forall f\in L^1_c(\Omega_1\cup\partial\Omega;\R^{n+1}),
\end{equation}
where $L^1_c(\Omega_1\cup\partial\Omega;\R^{n+1})$ denotes all the functions in $L^1(\Omega_1;\R^{n+1})$ with support in a compact subset of $\Omega_1\cup\partial\Omega$.

\begin{Definition}\label{weak} A function $u\in C^{0,1}(\Omega_1)$ defined as the limit of a sequence $\{u_{\varepsilon_k}\}$ of solutions to $(*_{\varepsilon_k})$, with $\e_k\downarrow 0$, as in \eqref{uetou} will be called a weak  solution of $(**)$. 
\end{Definition}
We have therefore established the following.
\begin{Theorem}\label{weak existence}
There exists $u\in C^{0,1}(\Omega_1\cup\partial\Omega)$ a weak solution of $(**)$, as in Definition~\ref{weak}, where $\Omega_1\stackrel{open}{\subset} (\Omega\setminus\ov\Omega_\out)$ and $\partial\Omega_1\supset\partial\Omega$ (recall that $\Omega_\out$ is such that $\partial\Omega_\out=\Sigma_\out$, the outermost MOTS). Furthermore, any weak solution satisfies $u|_{\partial\Omega}=0$ and $\lim_{x\to x_0} u(x)=+\infty$ for any $x_0\in\partial \Omega_1\setminus\partial\Omega$.
\end{Theorem}

\section{Properties of weak solutions}\label{properties}

In this section we study a weak solution $u\in C^{0,1}(\Omega_1)$  of $(**)$ (see Definition~\ref{weak}, Theorem~\ref{weak existence}), using a sequence   $\{u_{\e_k}\}$  of solutions to the problems $(*_{\varepsilon_k})$ such that
\[ u_{\varepsilon_k}\to u \text{  in $C^0({\Omega_1}\cup\partial\Omega)$}.\]
We will show a minimization property for the graphs of the  functions $u_{\e_k}$ (Lemma~\ref{emin}) and show that this property passes to the limit, i.e. it passes to $\graph u$ (Lemma~\ref{eminu}). We will also examine in more detail the convergence $u_{\e_k}\to u$ (Lemma~\ref{mainconvergence}) in order to study the part of the  boundary of $\Omega_1$ where $u$ blows up, as our goal is to show that it is a generalized MOTS. Many of the arguments in this section follow those of \cite{MS} and \cite{S08}, where the corresponding results are proven for the mean curvature flow and in \cite{S08} also for general  speeds given by powers of the mean curvature (the $H^k$-flow). In \cite{MS} and \cite{S08} the ambient space where the flows are considered is the Euclidean space (in \cite{S08} manifolds that do not contain closed minimal surfaces are also considered), therefore the corresponding `approximating' functions $u_\e$ are bounded. In our case, the functions $u_\e$ have a `blow up' set which causes an extra complication.

We first prove a uniform integral estimate for the right-hand side of  the equation $(*_{{\varepsilon}})$.
\begin{Lemma}\label{normalintest}
Let $u_{{\varepsilon}}\in C^\infty(\Omega_0)$ be a solution of $(*_{{\varepsilon}})$ as  in  Theorem~\ref{blowup}. Then
\begin{equation}\label{intest}
\int_{\Omega_0}\frac{1}{\sqrt{\e^2+|\nabla{u}_{\varepsilon}|^2}}dx\leq|\partial\Omega_0|+(n+2)\lambda|\Omega|,
\end{equation}
where $\lambda=\max_i\{|\l_i|, \l_i\text{  eigenvalue of  } K\}$.
\end{Lemma}
\begin{Remark}\label{rmk on normalintest}
Note that $|\partial\Omega_0|=|\partial\Omega|+|\Sigma_\out|$, where $\Sigma_\out$ is the outermost MOTS. In \cite{AM09} an estimate, in terms of the initial data, on $|\Sigma_\infty|$ is derived and therefore $|\partial\Omega_0|$ depends only on the initial data and $\Omega$.
\end{Remark}
\begin{proof}[Proof of Lemma \ref{normalintest}]
This follows as in \cite[Lemma 2.1]{MS}, keeping track of the extra ``$P$-term''. 
Let $\psi$ be a smooth function such that $0\leq\psi\leq1$, $\psi=1$ on $\Omega_{\delta}:=\{x\in\Omega_0|\text{dist}(x,\partial\Omega_0)>\delta\}$, $\psi=0$ on $\partial\Omega$ and $|D\psi|\leq\gamma/\delta$ for some $\gamma>1,\,\delta>0$. Multiplying $(*_{\varepsilon})$ by $\psi$ and integrating by parts we find
\begin{equation*}\begin{split}
\int_{\Omega_\d}\frac{1 }{\sqrt{\varepsilon^2+|\nabla u_{\varepsilon}|^2}}dx\leq \int_{\Omega_0\setminus\Omega_{\delta}}\frac{\nabla \psi\cdot \nabla u_{\varepsilon}}{\sqrt{\varepsilon^2+|\nabla u_{\varepsilon}|^2}}dx+\int_{\Omega_0}P(u_\e)\psi dx,
\end{split}\end{equation*}
where $P(u_\e)=\left(g^{ij}-\frac{\nabla^i u_\e\nabla^j u_\e}{|\nabla u_\e|^2+\e^2}\right) K_{ij}$.
Since $|P(u_\e)|\le (n+2)\l$, we have
\[\int_{\Omega_\d}\frac{1 }{\sqrt{\varepsilon^2+|\nabla u_{\varepsilon}|^2}}dx\leq\frac{\gamma}{\delta}|\Omega_0\setminus\Omega_{\delta}|+(n+2)\lambda|\Omega|\]
and after letting $\d\to 0$ and then $\gamma\to 1$ we obtain the result.
\end{proof}

Lemma~\ref{normalintest} and  the convergence of $\nabla u_{\e_k}$ given in \eqref{Duetou}, along with \cite[Theorems 3.1, 3.2 and 3.3]{ES4}, yield the following.

\begin{Lemma}\label{ES conv} Let $u\in C^{0,1}(\Omega_1)$ be a weak solution of $(**)$ and  $\{u_{\e_k}\}$ be a sequence of solutions to the problems $(*_{\varepsilon_k})$ such that
$ u_{\varepsilon_k}\to u$  in $C^0({\Omega_1}\cup\partial\Omega)$,
as in Definition~\ref{weak}. Then,  the following convergences are true.
\begin{itemize}
\item[(i)] 
$\int_{\Omega_1}|\nabla u_{\e_k}| fd\H^n\to \int_{\Omega_1}|\nabla u| fd\H^n\,,\,\,\forall f\in L^1_c(\Omega_1\cup\partial\Omega)$,

\item[(ii)]$\frac{\nabla u_{\e_k}}{\sqrt{\e_k^2+|\nabla u_{\e_k}|^2}}\to\frac {\nabla u}{|\nabla u|}$ \quad strongly in $L^2_\loc(\Omega_1\cap \{|\nabla u|\ne 0\};\R^{n+1})$.
\end{itemize}
\end{Lemma}
\begin{proof} The proof is exactly the same as that of \cite[Theorems 3.1, 3.2 and 3.3]{ES4}, with the difference that here we should substitute the domain of definition of all the functions (which is $\R^n$ in \cite{ES4}) with  $\Omega_1\subset M$. This change leaves the proof unaltered, provided that the test functions used are taken to be in $C^\infty_c(\Omega_1\cup \partial\Omega)$, instead of  $C^\infty_c(\R^n)$. We also point out  that hypothesis (3.2) used in \cite{ES4} should be replaced here with the  convergence $ u_{\varepsilon_k}\to u$  in $C^0({\Omega_1}\cup\partial\Omega)$ and that of $\nabla u_{\e_k}$ given in  \eqref{Duetou}, and hypothesis (3.5) used in \cite{ES4} is still true in our case because of Lemma~\ref{normalintest}, equation $(*_{\e_k})$ and the fact that $P$ is bounded. Finally, we remark that the result in \cite[Theorem 3.2]{ES4} is an intermediate step towards proving  \cite[Theorem 3.3]{ES4} (which corresponds to (ii) here), which in our case is replaced by
$\int_{\Omega_1}\frac{\nabla u_{\e_k}}{\sqrt{\e_k^2+|\nabla u_{\e_k}|^2}}fd\H^n\to \int_{\Omega_1}\frac {\nabla u}{|\nabla u|} fd\H^n$
for all  $f\in L^1(\Omega_1;\R^{n+1})$ with compact support in $(\Omega_1\cup\partial\Omega)\cap\{|\nabla u|>0\}$.
\end{proof}

Using now Lemma~\ref{normalintest}, together with the convergence (i) of Lemma~\ref{ES conv}, yields the following.
\begin{Lemma} \label{nonfaten}  Let $u\in C^{0,1}(\Omega_1)$ be a weak solution of $(**)$, as in Definition~ \ref{weak}, then $\H^{n+1}(\{x\in \Omega_1||\nabla u|= 0)=0$. 
\end{Lemma}
\begin{proof}
The proof is exactly the same as that of \cite[Lemma 2.3]{MS}, replacing  $\e_i$ and $\Omega$ with $\e_k$  and $\Omega_1\subset M$ respectively and the set $A$ with $A\cap W=\{x\in\Omega_1\cap W:Du(x)=0\}$ for any $W\subset\subset\Omega_1\cup\partial\Omega$. By the proof of   \cite[Lemma 2.3]{MS}, we then obtain that $\H^{n+1}(A\cap W)=0$ for any $W\subset\subset\Omega_1\cup\partial\Omega$ and thus the result follows.
\end{proof}

\begin{Remark}\label{ES conv2} Lemma~\ref{nonfaten} and Lemma~\ref{ES conv} (ii) imply that
$\frac{\nabla u_{\e_k}}{\sqrt{\e_k^2+|\nabla u_{\e_k}|^2}}\to\frac {\nabla u}{|\nabla u|}$ strongly in $L^2_\loc(\Omega_1;\R^n)$.
\end{Remark}

\begin{Definition}\label{Ue}
For a solution $u_{\e}\in C^\infty(\Omega_0)$ of $(*_{\e})$ we define  the function $U_{\e}:\Omega_0\times\R\to \R$ by $U_{\e}(x,z)=u_{\e}(x)-\e z$ and we let
\[\wt E_t^\e=\{(x,z)\in \Omega_0\times\R: U_{\e}(x,z)>t\}\]
and
\[\wt\Sigma^{\e}_t=\{(x,z)\in \Omega_0\times \R:U_{\e}(x,z)=t\}=\graph\left(\frac{u_{\e}}{\e}-\frac{t}{\e}\right),\]
the latter being the hypersurfaces given by the level sets of $U_\e$. 
\end{Definition}
As mentioned in the introduction (see \eqref{translating}), $\wt\Sigma^{\e}_t$ are smooth translating solutions of  the  null mean curvature flow $(*)$. We also note that, by equation $(*_{\hat \e})$,  the mean curvature of $\wt\Sigma^{\e}_t$ is given by
\begin{equation}\label{MCet}
H^{\e}_t=\dvg\nu_\e=-P(u_{\e})+\frac{1}{\sqrt{|\nabla u_{\e}|^2+{\e}^2}},
\end{equation}
where   $P(u_{\e})=\left(g^{ij}-\frac{\nabla^i u_{\e}\nabla^j u_{\e}}{|\nabla u_{\e}|^2+{\e}^2}\right) K_{ij}$ and $\nu_{\e}$ is the upward pointing unit normal to $\wt\Sigma^{\e}_t$. Note that $P(u_{\e})=\left(g^{ij}-\nu^i_{\e}\nu^j_{\e}\right) K_{ij}$, and thus we will also express this quantity as $P(\nu_{\e})$. Recall that $K$ and $\nu_\e$ are always extended in $M\times\R$ so that they are independent of the vertical component.

\begin{Lemma}\label{H+Pk} For any solution  $u_{\e}\in C^\infty(\Omega_0)$ of $(*_{\e})$, any $t\in \R$  and any interval $I=[a,b]\subset \R$ the graph  $\wt\Sigma^{\e}_t=\graph\left(\frac{u_{\e}}{\e}-\frac{t}{\e}\right)$ satisfies 
\[\int_{0}^{\infty}\int_{\wt\Sigma^{\e}_t\cap (\Omega_0\times I)}|H^\e_t+ P(\nu_\e)|^2d\H^{n+1}dt\le (b-a)\left(|\partial\Omega_0|+(n+2)\lambda|\Omega|
\right), \]
where  $ \lambda=\max_i\{|\l_i|, \l_i\text{  eigenvalue of  } K\}$ and the rest of the notation is as in Definition~\ref{Ue} and equation  \eqref{MCet}. 
\end{Lemma}

\begin{proof}
Using the coarea formula and the expression of $H^\e_t$ given in  \eqref{MCet}, we have 
\[\begin{split}\int_{\Omega_0\times I}\frac{1}{|\ov\nabla U_{\e}|}dx&=\int_{\Omega_0\times I}\frac{1}{|\ov \nabla U_{\e}|^2} |\ov\nabla U_{\e}|dx\\
&=\int_{0}^{\infty}\int_{\{(x,z):U_{\e}(x,z)=t\}\cap (\Omega_0\times I)}\frac{1}{|\ov \nabla U_{\e}|^2}d\H^{n+1}dt\\
&= \int_{0}^{\infty}\int_{\wt\Sigma^{\e}_t\cap (\Omega_0\times I)}|H^{\e}_t+ P(\nu_\e)|^2d\H^{n+1}dt,
\end{split}\]
where $\ov\nabla=\nabla^{M\times\R}$. The result now follows by  Lemma~\ref{normalintest}.
\end{proof}

Next we will show that the sets 
$\wt E_t^\e=\{U_{\e}>t\}$
(as in Definition~\ref{Ue}) minimize area plus {\it bulk energy P} on the outside in $\Omega_0\times\R$. More specifically, we have the following.

\begin{Lemma} \label{emin} For any solution  $u_{\e}\in C^\infty(\Omega_0)$ of $(*_{\e})$ and any $t\in \R$  the set  $\wt E^{\e}_t=\{U_\e>t\}$ satisfies  the following minimization property.
\[|\partial^*\wt E_t^\e\cap W|+\int_{W\cap \wt E_t^\e} P(\nu_k)d\H^{n+2}\le |\partial^*F\cap W|+\int_{W\cap F} P(\nu_\e)d\H^{n+2}\]
for any compact set $W\subset \Omega\times\R$ and any finite perimeter set $F$ with $\wt E_t^\e\subset F$ and $F\setminus \wt E_t^\e\subset W$. Here, again we use the notation in Definition~\ref{Ue} and equation \eqref{MCet}.
\end{Lemma}
\begin{proof}
Let $W$ and $F$ be as in the statement of the lemma and note that $F\setminus\wt E_t^\e\subset\Omega_0\times\R$.
By $(*_{\hat \e})$ (see also \eqref{MCet}), we have that
\[\dvg \nu_\e= -P(\nu_\e)+\frac{1}{|\ov \nabla  U_{\e}|}.\]
 The divergence theorem, using $\nu_\e$ as a calibration, yields
\[\begin{split}\int_{F\setminus \wt E^\e_t}-P(\nu_\e)+\frac{1}{|\ov \nabla  U_{\e}|}d\H^{n+2}=&-\int_{\partial^* \wt E_t^\e\cap W}\nu_\e\cdot \nu_{\partial^* \wt E_t^\e}d\H^{n+1}\\
&+ \int_{\partial^* F\cap W}\nu_\e\cdot \nu_{\partial^* F}d\H^{n+1}\\
&\le-|\partial^*\wt E_t^\e\cap W|+|\partial^*F\cap W|,
\end{split}\]
where $\nu_{\partial^* \wt E_t^\e}$ and $\nu_{\partial^* F}$ denote the outward pointing unit normals to $\partial^* \wt E_t^\e$ and ${\partial^* F}$ respectively. Using this, along with the fact that $\frac{1}{|D U_{\e}|}>0$, we have
\[\begin{split}
|\partial^*\wt E_t^\e\cap W|+\int_{W\cap \wt E_t^\e} P(\nu_\e)d\H^{n+2}=&|\partial^*\wt E_t^\e\cap W|-\int_{F\setminus \wt E^\e_t} P(\nu_\e)d\H^{n+2}\\
&+\int_{W\cap F} P(\nu_\e)d\H^{n+2}\\
\le& |\partial^*F\cap W| +\int_{W\cap F} P(\nu_\e)d\H^{n+2}.
\end{split}\]
\end{proof}

\begin{Remark}\label{area bound}
Lemma~\ref{emin} provides  a local uniform  area bound for $\partial^*\wt E^\e_t=\wt\Sigma^\e_t$  in $\Omega\times\R$ (since $K$, and thus $P$, is bounded). 
\end{Remark}

\begin{Remark}\label{minimrmk}
Arguing similarly to the proof of Lemma~\ref{emin}, it is not hard to show that the sets $\wt E_t^\e$ actually minimize (not only on the outside) the following
\[|\partial^*\wt E_t^\e\cap W|+\int_{W\cap \wt E_t^\e} P(\nu_\e)-\frac{1}{|\ov \nabla U_{\e}|}d\H^{n+2}.\]
However, this will not be needed in this paper.
\end{Remark}

We will now focus on a sequence  of  solutions to the problems $(*_{\varepsilon_k})$   that converge to a weak solution of  $(**)$.
\begin {Definition}\label{U} Let $u\in C^{0,1}(\Omega_1)$ be a weak solution of $(**)$ and  $\{u_{\e_k}\}\subset C^\infty(\Omega_0)$ be a sequence of solutions to the problems $(*_{\varepsilon_k})$ such that
$ u_{\varepsilon_k}\to u$  in $C^0({\Omega_1}\cup\partial\Omega)$,
as in Definition~\ref{weak} (see also Theorem~\ref{weak existence}). We define the function $U:\Omega_1\times\R\to\R$ by
\[U(x,z)=u(x).\]
Note that $U\in C^{0,1}(\Omega_1\times\R)$ and $U_{\e_k}\to U$ in $C^0(({\Omega_1}\cup\partial\Omega)\times\R)$, where the functions  $U_{\e_k}$ are as in Definition~\ref{Ue}. We, 
furthermore, set 
\[\wt E_t=\{(x,z)\in \Omega_1\times\R: U(x,z)>t\}= E_t\times\R\,,\,\,E_t=\{x\in \Omega_1:u(x)>t\}\]
and
\[\wt\Sigma_t=\partial \wt E_t=\Sigma_t\times\R\,,\,\,\Sigma_t=\partial E_t.\]
Finally,  for notational simplicity, the sets $\wt E_t^{\e_k}$ and $\wt\Sigma_t^{\e_k}$, as defined in Definition~\ref{Ue}, will be denoted by $\wt E_t^k$ and $\wt\Sigma_t^k$ respectively. Moreover, the upward pointing unit normal to $\wt\Sigma _t^k$ and its mean curvature will be denoted by $\nu_k$ and $H_t^k$ respectively, so that equation \eqref{MCet} now reads
\begin{equation}\label{MCet2}
H_t^k=\dvg\nu_k=-P(u_{\e_k})+\frac{1}{\sqrt{|\nabla u_{\e_k}|^2+1}},
\end{equation}
where recall that $P(u_{\e_k})=P(\nu_k)=(g^{ij}-\nu_k^i\nu_k^j)K_{ij}$.
\end{Definition}
 
We next want to show that the minimizing property of $\wt E_t^k$, described in Lemma~\ref{emin}, is also true for the limit $\wt E_t$. We first show that a weak solution $u$ is {\it non-fattening}, which will in turn imply that $E_t^k\to \wt E_t$ in $L^1_\loc$ for {\it all} $t>0$ (the convergence here should be understood as convergence in $L^1_\loc(\Omega_1\times\R)$ of the corresponding characteristic functions). More specifically, we have the following.
\begin{Lemma}\label{nonfat} Let $u\in C^{0,1}(\Omega_1)$ be a weak solution of $(**)$. Then, for all $t>0$ $\H^{n+1}(\{u=t\})=0$.
\end{Lemma}
\begin{proof}
The proof is exactly as that of \cite[Lemma 5.5]{S08}, where the same result is proven in the case $P=0$. We repeat the main step here and sketch the rest of the proof, using the notation of Definition~\ref{U}. Let $0<t_1<t_2$ and $\Omega_1'=\Omega_1\times I$ for some interval $I=(a,b)$. By using the coarea formula, \eqref{MCet2}, H\"older's inequality, Remark~\ref{area bound}, and Lemma~\ref{H+Pk}, we obtain
\[\begin{split}
|\H^{n+2}(\wt E_{t_1}^k\cap\Omega_1')-\H^{n+2}(\wt E_{t_2}^k\cap\Omega_1')|&=\int_{t_1}^{t_2}\int_{\wt\Sigma^k_t\cap\Omega_1'}|H_t^k+P(\nu_k)|\\
&\le C|t_2-t_1|^\frac12,
\end{split}\]
where $C$ is a constant independent on $k$. 

Let now $S=\{t>0:\H^{n+2}\{U=t\}>0\}=\{t>0:\H^{n+1}\{u=t\}>0\}$ and note that for any $t\notin S$ $\wt E_t^k\to\wt E_t$ in $L^1_\loc$ (in the sense that their characteristic functions converge in $L^1_\loc(\Omega_1\times\R)$), because of the local uniform convergence $U_{\e_k}\to U$ in $\Omega_1\times\R$. Thus, for $t_1,t_2\notin S$, the limit of the above estimate yields
\[|\H^{n+2}(\wt E_{t_1}\cap\Omega_1')-\H^{n+2}(\wt E_{t_2}\cap\Omega_1')|\le C|t_2-t_1|^\frac12.\]
For any $t\in S$ (a countable set) we can now pick two sequences of times $t_1^j<t<t_2^j$ for which the above is true and such that both sequences  tend to $t$. This then implies that $\H^{n+2}\{U=t\}=0$.
\end{proof}

\begin{Remark}\label{setconv} Lemma  \ref{nonfat}, along with the uniform convergence $U_{\e_k}\to U$, implies that $\wt E_t^k\to \wt E_t$ in $L^1_\loc(\Omega_1\times\R)$ (i.e. the corresponding characteristic functions converge in $L^1_\loc(\Omega_1\times\R)$, see also proof of \cite[Lemma 5.5]{S08}).
\end{Remark}
We are ready now to show that the minimizing property of $\wt E^k_t$, as presented in  Lemma~\ref{emin}, passes to the limit. More specifically,  we  show  that the sets $\wt E_t$ minimize area plus {\it bulk energy P} on the outside in $\Omega_1\times\R$. The same is then also true for the sets $E_t$ in $\Omega_1$.  

 \begin{Lemma} \label{eminu} For any weak solution  $u\in C^{0,1}(\Omega_1)$ of $(**)$ and any $t\in \R$  the set  $\wt E_t=\{U>t\}$ satisfies  the following minimization property.
\[|\partial^*\wt E_t\cap W|+\int_{W\cap \wt E_t} P(\nu)d\H^{n+2}\le |\partial^*F\cap W|+\int_{W\cap F} P(\nu)d\H^{n+2}\]
for any compact set $W\subset \Omega\times\R$ and any finite perimeter set $F$ with $\wt E_t\subset F$ and $F\setminus \wt E_t\subset W$.  
Here, $P(\nu)=\left(g^{ij}-\nu^i\nu^j\right) K_{ij}$ and $\nu$ is independent of the vertical component with $\nu(x,z)= \nu(x)=\frac{Du}{|Du|}(x)$ $\H^{n+1}$-a.e. on $\Omega_1$, and we use the notation in Definition~\ref{U}.. 

Furthermore, the same minimizing property is satisfied by $E_t$ in $\Omega_1$, that is
\[\begin{split}
|\partial^* E_t\cap W|+\int_{ W\cap  E_t} P(\nu)d\H^{n+1}\le |\partial^* F\cap  W|+ \int_{W\cap  F} P(\nu)d\H^{n+1}
\end{split}\]
for any compact set $W\subset\Omega$ and any finite perimeter set $F$ with $ E_t\subset F$ and $F\setminus E_t\subset W$.
\end{Lemma}

\begin{proof} The proof follows that of \cite[Lemma 5.6, Corollary 5.7]{S08}, where the same statements are proven in the case when $P=0$. Let $W$, $F$ be as in the statement of the lemma and note that $F\setminus\wt E_t\subset\Omega_1\times\R$. First, note that arguing exactly as in \cite[Lemma 5.6]{S08}, we can assume, by passing to a slightly larger compact set if necessary, that for $W$  the following is true. The boundary  $\partial W$ is smooth, 
$|\partial^*(F\cup \wt E_t^k)\cap\partial W|=|\partial^*(F\cap \wt E_t^k)\cap\partial W|=|\partial^* \wt E_t^k\cap\partial W|=0$ for all $k$
 and $\lim_{k\to\infty}\int_{\partial W} |\phi^-_{F\cup \wt E_t^k}-\phi^+_{\wt E_t^k}|d\H^{n+1}= 0$,
where $\phi^-_{F\cup \wt E_t^k}$ and $\phi^+_{\wt E_t^k}$
are the inner and outer trace of ${F\cup \wt E_t^k}$ and  $\wt E_t^k$ on $\partial W$ (see \cite[Chapter 2]{Giu} and \cite[Section 2.4]{Alex} for definitions of the traces and note that here we also use Remark~\ref{setconv}). Let now $F^k= \wt E_t^k\cup (F\cap W)$. We  then have
\[\begin{split}|\partial^*F^k\cap W|=&\int_{\partial W} |\phi^-_{F\cup \wt E_t^k}-\phi^+_{\wt E_t^k}|d\H^{n+1}+|\partial^*(F\cup \wt E_t^k)\cap W|\\
=&\int_{\partial W} |\phi^-_{F\cup \wt E_t^k}-\phi^+_{\wt E_t^k}|d\H^{n+1}+ |\partial^* \wt E^k_t\cap W|\\
&+ |\partial^* F\cap W|-|\partial^*(F\cap \wt E_t^k)\cap W|,
\end{split}\]
where the second equality above is justified by  arguing as in \cite[(36) of proof Lemma~5.6]{S08}).
By the minimizing property of  $\wt E_t^k$ (since $F^k\supset \wt E_t^k$), we have
\[|\partial^* \wt E^k_t\cap W|-\int_{W\cap (F^k\setminus \wt E^k_t)} P(\nu_k)d\H^{n+2}\le |\partial^* F^k\cap W|,\]
and thus  we  obtain 
\[\begin{split}
 |\partial^* F\cap W|\ge &|\partial^*(F\cap \wt E_t^k)\cap W|- \int_{W\cap (F^k\setminus \wt E^k_t)} P(\nu_k)d\H^{n+2}\\
 &-\int_{\partial W} |\phi^-_{F\cup \wt E_t^k}-\phi^+_{\wt E_t^k}|d\H^{n+1}.
\end{split}\]
Since the last term on the right-hand side vanishes as $k\to \infty$ and $|\partial^* \wt E_t\cap W|=|\partial^*(F\cap \wt E_t)\cap W| \le \lim_{k\to\infty}|\partial^*(F\cap\wt E_t^k)\cap W|$ (by Remark~\ref{setconv} and the lower semi-continuity, see \cite[Theorem 1.9]{Giu} and \cite[Theorem 2.38]{Alex}), it suffices to show that
\[\int_{W\cap (F^k\setminus \wt E^k_t)} P(\nu_k)d\H^{n+2} \stackrel{k\to \infty}{\longrightarrow}\int_{W\cap (F\setminus \wt E_t)} P(\nu)d\H^{n+2}, \]
where $\nu$ is as in the statement of the lemma. To see this, we note that
$W\cap (F^k\setminus \wt E^k_t)=(F\cap W)\setminus\wt E_t^k$ and $ W\cap (F\setminus \wt E_t)=(F\cap W)\setminus \wt E_t$, and 
we  write
\[\begin{split}\int_{W\cap (F^k\setminus \wt E^k_t)} &P(\nu_k)d\H^{n+2}-\int_{W\cap (F\setminus \wt E_t)} P(\nu)d\H^{n+2}
\\
=& \int_{W\cap F\cap(\wt E_t\setminus\wt E_t^k)} P(\nu_k)d\H^{n+2}-\int_{W\cap F\cap(\wt E^k_t\setminus\wt E_t)} P(\nu_k)d\H^{n+2}\\
&+ \int_{W\cap (F\setminus \wt E_t)} P(\nu_k)-P(\nu)d\H^{n+2}.
\end{split}\]
We can see now that the right-hand side of the above equality tends to $0$, as $k\to \infty$, because of the fact that $P$ is bounded, Remark~\ref{setconv} and Lemma~\ref{ES conv}. More specifically, we have the following two observations. First, by Remark~\ref{setconv}, we have
\[\H^{n+2}(W\cap (\wt E_t\setminus \wt E^k_t))\,,\,\, \H^{n+2}(W\cap (\wt E_t^k\setminus \wt E_t))\stackrel{k\to \infty}{\longrightarrow} 0,\]
which implies that the first two terms tend to zero. Second, by Lemma~\ref{ES conv} (see also Remark~\ref{ES conv2}), we  have
\[\int_{W\cap (F\setminus \wt E_t)} \nu_k^i-\nu^id\H^{n+2}\le C(W)\|\nu_k-\nu\|_{L^2(W)} \stackrel{k\to \infty}{\longrightarrow}0,\]
and thus, by writing $\nu_k^i\nu_k^j-\nu^i\nu^j=\nu_k^i(\nu_k^j-\nu^j)+\nu^j(\nu_k^i-\nu^i)$, we have
\[\int_{W\cap (F\setminus \wt E_t)} \nu_k^i\nu_k^j-\nu^i\nu^jd\H^{n+2}\stackrel{k\to \infty}{\longrightarrow}0.\]
Since $P(\nu_k)- P(\nu)= (\nu_k^i\nu_k^j-\nu^i\nu^j) K_{ij}$, this implies that the last term also tends to zero.

Finally, one can easily see that the same minimization property holds for $E_t$ in $\Omega_1$ as follows (cf. \cite[Corollary 5.7]{S08}). Let $W\subset \Omega$ be a compact set and let $F$ be a finite perimeter set such that $E_t\subset F$ and $F\setminus E_t\subset W$. Given any $\ell>0$, let $\wt F=(F\times(-\ell,\ell))\cup \wt E_t$. Using the minimization property of  $\wt E_t$, we have
\[|\partial^*\wt E_t\cap \wt W|+\int_{\wt W\cap \wt E_t} P(\nu)d\H^{n+2}\le |\partial^*\wt F\cap \wt W|+\int_{\wt W\cap \wt F} P(\nu)d\H^{n+2},\]
where $\wt W=W\times[-2\ell, 2\ell]$. This then yields
\[\begin{split}
2\ell|\partial^* E_t\cap W|+2\ell \int_{ W\cap  E_t} P(\nu)d\H^{n+1}\le& 2\ell|\partial^* F\cap  W|+2\H^{n+1}(F\setminus E_t)\\
&+2\ell \int_{W\cap  F} P(\nu)d\H^{n+1}.
\end{split}\]
Dividing by $\ell$ and letting $\ell\to \infty$ provides the required property.
\end{proof}

 We now  define the measures
 \begin{equation*}\label{measdef}
 \mu_t^k=\H^{n+1}\res \wt \Sigma_t^k \text{  and  }\mu_t=\H^{n+1}\res \partial^*\wt E_t,
 \end{equation*}
 where recall that $\wt \Sigma_t^k=\partial(\{(x, z): U_{\e_k}(x,z)>t\})=\graph \left(\frac{u_{\e_k}}{\e_k}-\frac{t}{\e_k}\right)$ (see Definition~\ref{U}).
 Our goal is to show that $\mu_t^k\to \mu_t$ as Radon measures. This is done following the steps in \cite[Section 5]{S08}. We first show that the sets $\partial^* E_t\subset \Sigma_t:=\partial E_t\subset \{x\in \Omega_1:u(x)=t\}$ are equal up to a set of  $\H^n$-measure zero. 
\begin{Lemma}\label{bdrystar} Let  $u\in C^{0,1}(\Omega_1)$ be a weak solution of $(**)$. Then, for  a.e. $t\in[0,\infty)$, $\H^n(\{u=t\}\setminus\partial^*\{u>t\})=0$
\end{Lemma}
\begin{proof}
This is proven exactly as \cite[Lemma 5.9]{S08}.  Since $u\in C^{0,1}(\Omega)\subset \BV(\Omega)$ we can compare the coarea formula for BV-functions and Lipschitz functions to obtain
\[\int_0^T\H^n(\partial^*E_t)dt=\int_{\Omega_1\cap\{u<T\}}|\nabla u|d\H^{n+1}=\int_0^T\H^n(\{u=t\})dt\]
for any $T>0$, and thus
\[\int_0^T\H^n(\{u=t\}\setminus\partial^* E_t)dt=0,\]
which yields that $\H^n(\{u=t\}\setminus\partial^* E_t)=0$ for a.e. $t\in [0,T]$. Since this is true for all $T>0$, we obtain the result.
\end{proof}

We are now ready to prove the measure convergence.

\begin{Lemma}\label{mainconvergence} Let  $\mu_t^k=\H^{n+1}\res \wt \Sigma_t^k$ and $\mu_t=\H^{n+1}\res \partial^*\wt E_t$ (where we use the notation of Definition~\ref{U}). Then, for a.e  $t>0$  $\mu_t^k\stackrel{k\to\infty}{\longrightarrow}\mu_t$ as Radon measures.
\end{Lemma}

\begin{proof}
The proof is almost identical to that of \cite[Proposition 5.10]{S08}. We go through the proof here pointing out the differences in our case. To fit our notation, one has to replace $i$, $\e_i$, $\Omega$, $E_t^i$, $N_t^i$, $E'_t$ and $\Gamma_t$  of \cite[Proposition 5.10]{S08} by $k$, $\e_k$, $\Omega_0$, $\wt E_t^k$, $\wt E_t$, $\wt \Sigma_t^k$ and $\Sigma_t$ respectively.
Note first that, by Lemma~\ref{bdrystar}, for almost every $t$   $\mu_t= \H^{n+1}\res \wt \Sigma_t$, where recall that $\wt \Sigma_t=\partial \wt E_t=\Sigma_t\times\R=\partial E_t\times\R$.
Fix a $t>0$, so that the above is true. By the minimizing property (Lemma~\ref{emin}, Remark~\ref{area bound}), $|\wt \Sigma^k_t|$ are locally uniformly bounded and thus, after passing to a subsequence,  $\mu_t^k\stackrel{k\to\infty}{\longrightarrow}\mu$, where $\mu$ is a Radon measure in $\Omega_0\times\R$ (note that here we keep the same notation for the subsequence, whereas in \cite[Proposition 5.10]{S08} the subsequence is denoted by $\{\mu_t^{i_j}\}_j\subset\{\mu_t^i\}_i$, so to fit our notation one has to further replace $i_j$ of \cite[Proposition 5.10]{S08} by $k$).\\
{\bf Claim 1:} $\spt \mu\subset\{u=t\}\times\R$.\\
The proof of Claim 1 is identical to that of \cite[Claim 1 of proof of Proposition 5.10]{S08}.\\
{\bf Claim 2:} For $B_\r(x)\subset\subset\Omega_0\times\R$
\[\mu(\ov B_\r(x))\le \H^{n+1}(\partial B_\r(x))+C(K)\H^{n+2}(B_\r(x)).\]
The proof of Claim 2 is the same as  that of \cite[Claim 2 of proof of Proposition 5.10]{S08}, with the only difference being the bound for $\mu^{k}_t(B_\r(x))=\H^{n+1}(\partial^* \wt E_t^k\cap B_\r(x))$. In particular, here using Lemma~\ref{emin} with $F= \wt E_t^k\cup B_\r(x)$,  we obtain
\[\mu^{k}_t(B_\r(x))+\int_{\wt E_t^k\cap B_\r(x)} P(\nu_k)d\H^{n+2}\le  \H^{n+1}( \partial B_\r(x))+\int_{ B_\r(x)} P(\nu_k)d\H^{n+2},\]
which yields
\[ \mu^{k}_t(B_\r(x))\le \H^{n+1}( \partial B_\r(x))+C(K)\H^{n+2}( B_\r(x)).\]

We have then, as in  \cite[Claim 2 of proof of Proposition 5.10]{S08}, that $\mu$ is absolutely continuous with respect to  the $\H^{n+1}$-measure (since the $\H^{n+2}$-measure  is absolutely continuous with respect to  the $\H^{n+1}$-measure). Thus, by the Radon-Nikodym theorem, Claim 1, and Lemma~\ref{bdrystar}, we obtain that  there exists a function $\theta\in L^\infty(\Sigma_t\times\R, \H^{n+1})$ such that
\begin{equation}\label{mumut}
\mu=(\H^{n+1}\res(\partial^*E_t\times\R))\res\theta=\mu_t\res\theta,
\end{equation}
where recall that $\Sigma_t=\partial E_t$.

\noindent{\bf Claim 3:} $\theta\ge 1$ $\H^{n+1}$-a.e. on $\Sigma_t\times\R$.\\
The proof of Claim 3 is identical to that of \cite[Claim 3 of proof of Proposition 5.10]{S08}.\\
%
{\bf Claim 4:} $\theta\le 1$ $\H^{n+1}$-a.e. on $\partial^*E_t$.\\
The proof of Claim 4 is the same as that of \cite[Claim 4 of proof of Proposition 5.10]{S08} with the only difference being the way we obtain the  bound for $\mu_\l^{k}(B_1)= \l^{-(n+2)}\mu_t^{k}(\l B_1)$ (denoted as $\mu_\l^{i_j}$ in \cite[line 20, page 221]{S08}). In particular, here, one has to use the  minimizing  property of $E_t^k$ given in Lemma~\ref{emin} (whereas in \cite{S08} $E_t^k$ are minimizing area on the outside). This, however, does not change the argument as it only changes the bound by a term of order $\e$. More specifically, the term $\sup_{S_{2\e}\cap B_1}|P(\nu_k)||S_{2\e}\cap B_1|$ should be added to the bound. Here, $P(\nu_\k)$ is the term appearing in Lemma~\ref{emin} and $S_{2\e}=\{x\in \R^{n+2}:|x_{n+2}|\le 2\e\}$ as in \cite{S08}. We remark also that in the proof of this claim, one uses ``rescalings'' of sets in $M\times\R$ and of the measures $\mu_t^k, \mu_t$, which are defined via the exponential map. In particular one makes the identifications
\[(B_\r(x), g)\simeq(B_\r(\vec 0), \exp^*_x g)\simeq (B_\r(0), \hat g_{ij}^\phi)\]
as Riemannian manifolds, where $(B_\r(x), \phi)$ are
the geodesic normal coordinates that correspond to the identification $(T_pM, g(p))\simeq(R^{n+2}, \langle\cdot,\cdot\rangle)$ as Hilbert spaces and $\hat g_{ij}^\phi=g_{ij}^\phi\circ\phi^{-1}$, where $\hat g_{ij}^\phi$
are the components of $g$ in geodesic normal coordinates. These identifications allow us to reduce the proof to the case that $M=\R^{n+1}$.  We remark also that, since $\exp^{-1}$ is an isometry, the minimizing property given in Lemma~\ref{emin} and Claim 2 (both of which are used in the proof of this claim) are preserved under these identifications. For a detailed discussion and proofs of these facts see \cite{Alex}.

 Finally, Claim 3 and Claim 4 imply that $\theta(x)=1$ $\H^{n+1}$-a.e. on $\Sigma_t\times\R$. Thus, the limit measure $\mu$ does not depend on the subsequence and thus the whole sequence converges $\mu_t^k\to \mu_t$.
\end{proof}

Having established the measure convergence $\mu_t^k\stackrel{k\to\infty}{\longrightarrow}\mu_t$, or  $\H^{n+1}\res \wt \Sigma_t^k\to \H^{n+1}\res \partial^*\wt E_t= \H^{n+1}\res (\Sigma_t\times\R)$,
in Lemma~\ref{mainconvergence}, we would like to study now the limit of the measures $\mu_t=\H^{n+1}\res \Sigma_t\times\R$ as $t\to\infty$.

By Lemma~\ref{H+Pk}, we have the interval $I_0=(-1, 1)\subset \R$
\begin{equation}\label{24I0}
\int_{0}^{\infty}\int_{\wt\Sigma^{k}_t\cap (\Omega_0\times I_0)}|H^k_t+ P(\nu_k)|^2d\H^{n+1}dt\le C(\Omega, \Omega_0, K, n), 
\end{equation}
where $H^k_t$ is the mean curvature of $\wt\Sigma^{k}_t$, $P(\nu_k)=\left(g^{ij}-\nu^i_{k}\nu^j_{k}\right) K_{ij}$ and $\nu_{k}$ is the upward pointing unit normal to $\wt\Sigma^{k}_t$, as in Definition~\ref{U}. Hence, by Fatou's lemma, for almost every $t\in [0,\infty)$
\[
\begin{split}
\liminf_k&\int_{\wt\Sigma^{k}_t\cap (\Omega_0\times I_0)}|H^k_t+ P(\nu_k)|^2d\H^{n+1}\\
&=\liminf_k\int_{\Omega_0\times I_0}|H^k_t+ P(\nu_k)|^2d\mu_t^k<\infty. 
\end{split}
\]
Considering such a $t$, we conclude that 
there exists a subsequence $\{k_i\}_i$ (depending on $t$) such that
\begin{equation}\label{H+Pki}
\sup_{i\ge 0}\int_{\Omega_0\times I_0}|H_t^{k_i}+P(\nu_{k_i})|^2d\mu_t^{k_i}\le C.
\end{equation}
By \eqref{H+Pki}  and the triangle inequality, we obtain that for all $i\in\N$ and any compact subset $W\subset\Omega_0\times I_0$
\begin{equation}\label{Ht2bound}
\int_{(\Omega_0\times I_0)\cap W}|H_t^{k_i}|^2d\mu_t^{k_i}\le C( \Omega, K, n, W),
\end{equation}
where we have also used that $|P(\nu_{k_i})|$ is bounded and that, by Remark~\ref{area bound},  the graphs  $\wt \Sigma_t^{k_i}$ have uniformly bounded area in $W$ so that $\int_{(\Omega_0\times I_0)\cap W}d\mu_t^{k_i}\le C=C(|\Omega|, |\partial\Omega|,  \l)$, where  $\lambda=\max_i\{|\l_i|, \l_i\text{  eigenvalue of  } K\}$.

We have shown thus that $\wt\Sigma_t^{k_i}$
have locally uniformly  bounded in $L^2$ first variation in $\Omega_0\times I_0$, and since they also have locally uniformly bounded area,
we can apply the varifold compactness theorem of Allard \cite{Al}, which yields that,  after passing to a further subsequence, $\wt\Sigma_t^{k_i}$  converge in the varifold sense to an $(n+1)$-dimensional integral varifold. Since, by Lemma~\ref{mainconvergence}, $\mu_t^k\to \mu_t=\H^{n+1}\res(\Sigma_t\times\R)$,  we obtain that for a.e. $t>0$
\begin{equation}\label{varconv}
\wt\Sigma^{k_i}_t\to \Sigma_t\times I_0=\Sigma_t\times (-1,1)\,\,\,\text{in   }\Omega_0\times I_0=\Omega_0\times(-1,1),
\end{equation}
in the sense of varifolds, where $\Sigma_t\times I_0$ is an $(n+1)$-dimensional rectifiable unit density varifold.
The varifold convergence and ~\eqref{Ht2bound} implies that $\Sigma_t\times I_0$ carries a generalized mean curvature vector $\vec H_t$ and $\mu^{k_i}_t\res H_t^{k_i}\nu_{k_i}\to\mu_t\res \vec H_t$ as vector valued Radon measures.  By the lower semicontinuity of the first variation and \eqref{Ht2bound}, we have
\begin{equation}\label{Hlimit}
\int_{\Omega_0\times I_0}|\vec H_t|^2d\mu_t\le \liminf_i\int_{\Omega_0\times I_0}|H^{k_i}_t|^2d\mu^{k_i}_t\le C=C(\Omega, \Omega_0,K,n).
\end{equation}
Since $\Sigma_t\times I_0$ is a rectifiable unit density varifold, $\vec H_t$ is perpendicular to $\Sigma_t\times I_0$ $\H^{n+1}$-a.e. (see \cite[Chapter 5]{B}).

We further have that $\wt E_t^{k_i}\to \wt E_t$ as finite perimeter sets (see Remark~\ref{setconv}), and recall that 
$\partial \wt E^{k_i}_t=\partial^*  \wt E^{k_i}_t=\wt\Sigma^{k_i}_t$ and  $\H^{n+1}((\Sigma_t\times\R)\setminus\partial^*\wt E_t)=0$ (Lemma~\ref{bdrystar}).
This implies that $\mu^{k_i}_t\res\nu_{k_i}\to\mu_t\res  \nu$ as vector valued Radon measures, where $\nu$ is the measure theoretic outer pointing unit normal to $\wt E_t$. Recall now that $u\in C^{0,1}(\Omega_1)$ and by Lemma~\ref{nonfaten}, Lemma~\ref{bdrystar} and the coarea formula for lipschitz functions (which imply that for a.e. $t\ge0$ $|\nabla u|\ne 0$ $\H^{n+1}$-a.e. on $\Sigma_t\times\R$) we have that for almost every $t\ge 0$ $\nu=\nu(x,z)=\frac{\nabla u(x)}{|\nabla u(x)|}$ $\H^{n+1}$-a.e. on $\Sigma_t\times\R$ (cf. Lemma~\ref{ES conv}). Note also that for the generalized mean curvature vector $\vec H_t$, as above, we obtain
 $\vec H_t= H_t\nu$ $\H^{n+1}$-a.e. on $\Sigma_t\times I_0$.  
The convergence  $\mu^{k_i}_t\res\nu_{k_i}\to\mu_t\res  \nu$,  along with the measure convergence $\mu^{k_i}_t\to\mu_t$ (Lemma~\ref{mainconvergence} or \eqref{varconv}), implies, using the Reshetnyak continuity  \cite[Theorem 2.39]{AFP}, that $\mu^{k_i}_t\res  P(\nu_{k_i})\nu_{k_i}\to\mu_t\res  P(\nu) \nu$ as vector valued Radon measures. 
Finally, this last convergence, along with $\mu^{k_i}_t\res H_t^{k_i}\nu_{k_i}\to\mu_t\res \vec H_t=\mu_t\res H_t\nu$, and using the lower semicontinuity, yields
\begin{equation}\label{H+Plimit}
\begin{split}
 \int_{\Omega_0\times I_0}|H_t&+P(\nu)|d\mu_t  \le\liminf_i\int_{\Omega_0\times I_0}|H^{k_i}_t+ P(\nu_{k_i})|d\mu_t^{k_i}.
  \end{split}
\end{equation}

Recall that this holds for a.e. $t\in[0,+\infty)$. Hence, using \eqref{H+Plimit}, Fatou's lemma, \eqref{24I0} and the uniform area bounds of Remark~\ref{area bound},  we obtain
\begin{equation*} \int_0^\infty\int_{\Omega_0\times I_0}|H_t+P(\nu)|d\mu_tdt\le C(\Omega, \Omega_0, K, n).
 \end{equation*}
 Putting everything together we have the following (cf. \cite[Theorem 5.11]{S08})
\begin{Theorem}\label{mainconvii}
Let  $\mu_t^k=\H^{n+1}\res \wt \Sigma_t^k$ and $\mu_t=\H^{n+1}\res \partial^*\wt E_t$ (where we use the notation of Definition~\ref{U}). Then, for a.e  $t\ge 0$  there exists a subsequence $\{k_i\}_i$ (depending on $t$) such that 
\[\wt\Sigma^{k_i}_t\to \Sigma_t\times (-1,1)\,\,\,\text{in  }\Omega_0\times(-1,1),
\]
as varifolds, where $\Sigma_t\times (-1,1)$ is a rectifiable unit density varifold that carries a generalized mean curvature vector $\vec H_t= H_t\nu$, where $\nu=\nu(x,z)=\frac{\nabla u(x)}{|\nabla u(x)|}$ $\H^{n+1}$-a.e. on $\Sigma_t\times(-1,1)$ (and recall that $u\in C^{0,1}(\Omega_1)$ is the weak solution of $(**)$ as in Definition~\ref{U}). Furthermore, we have
\begin{equation*} 
\int_{\Omega_0\times (-1,1)}|H_t|^2d\mu_t\le C(\Omega, \Omega_0, K, n)
\end{equation*}
and
\[
\int_0^\infty\int_{\Omega_0\times (-1,1)}|H_t+P(\nu)|d\mu_tdt\le C(\Omega, \Omega_0, K, n)
\]
where $P(\nu)=(g^{ij}-\nu^i\nu^j)K_{ij}$.
\end{Theorem}

We note that, because of the product structure of the varifold $\Sigma_t\times (-1,1)$ in Theorem~\ref{mainconvii},  for the $n$-dimensional rectifiable unit density varifolds $\Sigma_t$ we have
\begin{equation}
\label{H+Ptlimi} 
\begin{split}
\int_0^\infty\int_{\Omega}|H_t+P(\nu)|d\mu_t&=\int_0^\infty\int_{\Omega_0}|H_t+P(\nu)|d\mu_t\le C(\Omega, \Omega_0, K, n), \\
\int_{\Omega_0\times (-1,1)}|H_t|^2d\mu_t&\le C(\Omega, \Omega_0, K, n),
\end{split}
\end{equation}
where now $\mu_t=\H^{n}\res\Sigma_t$,  $\vec H_t=H_t \nu$ is the generalized mean curvature vector of $\Sigma_t$ and for almost every $t\ge0$ $\nu=\nu(x)=\frac{\nabla u}{|\nabla u|}(x)$ $\H^n$-a.e. on $\Sigma_t$. (We keep the same notation, as from now on we will concentrate only on $M$ and  forget about the product structure $M\times\R$, and therefore there will not be any confusion). We now want to study $\Sigma_t$ as $t\to \infty$ and show that they converge, as finite perimeter sets,  to a generalized MOTS, as in Definition~\ref{weakMOTS}.

Estimate \eqref{H+Ptlimi} allows us to pick a sequence of times $t_i\uparrow \infty$ such that
\begin{equation}\label{to0}
\lim_{i\to\infty}\int_{\Omega}|H_{t_i}+P(\nu)|d\mu_{t_i}=0.
\end{equation}
By the minimizing property, Lemma~\ref{eminu} (see also Remark~\ref{area bound}), $|\Sigma_{t_i}|$ are uniformly bounded  (recall that $\H^n(\Sigma_{t_i}\setminus\partial^*E_{t_i})=0$ by Lemma~\ref{bdrystar}) and thus, after passing to a subsequence, $\mu_{t_i}\to \mu_\infty$, where $\mu_\infty$ is a Radon measure in $\Omega$. Furthermore, considering $E_{t_i}$ as finite perimeter sets and using the compactness for such sets, we obtain that, passing to a further subsequence, $E_{t_i}\to E_\infty$ (that is $\chi_{E_{t_i}}\to \chi_{E_\infty}$ with respect to the $L^1(\Omega)$ norm), where $E_\infty$ is a finite perimeter set in $\Omega$.
 Moreover, since $D\chi_{E_\infty}=(\H^{n}\res\partial^* E_\infty)\res \nu_\infty$ as vector valued measures, where $\nu_\infty$ is the measure theoretic outer pointing unit normal to $E_\infty$  (see for example \cite[Section 5.7]{EG92}), we have the convergence $\mu_{t_i}\res \nu\to(\H^n\res\partial^* E_\infty)\res\nu_\infty$. (The definitions and the theorems used in relation with the finite perimeter sets can be found for example in  \cite[Chapters 1 and 3]{Giu}, see also \cite{Alex} for the extensions of these results for finite perimeter sets in a manifold).  We claim now that one can argue as in Lemma~\ref{mainconvergence} to show that 
 $\mu_\infty\res\partial^*E_\infty=\H^n \res\partial^*E_\infty$. In particular we have the following
 \begin{Lemma}\label{mainconvII}
 Assume that $\mu_{t_i}=\H^n\res\Sigma_{t_i}\to\mu_\infty$ as Radon measures and $E_{t_i}\to E_\infty$ as finite perimeter sets (where we use the notation of Definition~\ref{U}). Then 
 \begin{equation}\label{muinfty}
\mu_\infty\res\partial^*E_\infty=\H^{n}\res \partial^* E_\infty.
\end{equation}
 \end{Lemma}
\begin{proof}
The proof is exactly as in  Claims 2-4 of the proof of Lemma~\ref{mainconvergence}, using now the measures  $\H^n\res\partial^*E_\infty$ and $\mu_\infty\res\partial^*E_\infty$ (instead of $\mu_t$ and $\mu$, see \eqref{mumut}) and therefore we will not repeat it here. We point out that, to fit the notation of this lemma, one has to replace  $\Omega_0, \mu_t^k, \wt E_t^k$ and $\nu_k$ (of Lemma~\ref{mainconvergence}) by $\Omega, \mu_{t_i}, E_{t_i}$ and $\nu$ respectively and also replace both $\Sigma_t$ and $\partial^*\wt E_t$ by $ \partial^* E_\infty$. We also remark that in the proof here we need to use Lemma~\ref{eminu} instead of Lemma~\ref{emin} (or rather its corollary given in Remark~\ref{minimrmk}) and the lower semicontinuity of finite perimeter sets (lower semicontinuity of BV functions) instead of that for Radon measures for the convergence $E_{t_i}\to E_\infty$.
\end{proof}

We claim now that $\H^n\res \partial^*E_\infty$ has a generalized  mean curvature $\vec H_\infty=H_\infty\nu_\infty$ and it furthermore  satisfies $H_\infty+P(\nu_\infty)=0$, where recall that $\nu_\infty$ is the measure theoretic outer pointing unit normal to $E_\infty$.
This will then imply that $\partial^*E_\infty$ is a generalized MOTS in the sense of Definition~\ref{weakMOTS}.
To this aim we will argue as with the convergence in \eqref{varconv} replacing now $\wt \Sigma^{k_i}_t$ by $\Sigma_{t_i}$. 

By the structure theorem for  finite perimeter sets (see for example \cite[Section 5.7]{EG92}), we know that for $\H^{n}$-a.e. $x\in\partial^*E_\infty$ there exists $B_r(x)\subset M$ so that $E_\infty\cap B_r(x)$ is $C^1$, that is 
\begin{equation}\label{C1mfld}
\partial^* E_\infty\cap B_r(x)=\partial E_\infty\cap B_r(x)\text{   is a   }C^1 \text{   manifold}.
\end{equation} 
Furthermore, as mentioned before, we have  $\|D\chi_{E_\infty}\|=\H^{n}\res\partial^* E_\infty$.

Note now, that \eqref{H+Ptlimi}, implies that
\begin{equation}\label{Htlimibound}
\sup_{i\ge 0}\int_{\Omega}|H_{t_i}|^2d\mu_{t_i}\le C.
\end{equation}
\eqref{Htlimibound} shows that $\Sigma_{t_i}$
have  uniformly  bounded in $L^2$ first variation in $\Omega$, and since they also have  uniformly bounded area, we can apply the varifold compactness theorem of Allard \cite{Al}. Therefore,  after passing to a  subsequence,  $\Sigma_{t_i}\to \Sigma_\infty$ in $\Omega$ in the sense of varifolds, where $\Sigma_\infty$ is an integral $n$-dimensional varifold in $\Omega$ which carries a weak mean curvature $\vec H_\infty$ for which the bound \eqref{Htlimibound} still holds. Furthermore, $\vec H_\infty$ is perpendicular to $\Sigma_\infty$ $\H^{n}$-a.e. (see \cite[Chapter 5]{B}) and $\mu_t\res(H_{t_i} \nu)\to \mu_\infty\res\vec H_\infty$, where $\mu_\infty$ is the weight measure of $\Sigma_\infty$. We can now relate the varifold limit $\Sigma_\infty$ with $\partial^*E_\infty$ (the limit of finite perimeter sets or currents) by using \cite{W08}. In particular, by \cite[Theorem 1.2]{W08}, $\Sigma_\infty$ and $\partial^*E_\infty$ are compatible, that is $\Sigma_\infty=\var(\partial^* E_\infty)+ 2V$, where $V$ is some integral varifold in $\Omega$ and $\var(\partial^* E_\infty)$ is the varifold determined by $\partial^*E_\infty$ (see \cite[\S27]{LSgmt}). Using this, \eqref{C1mfld} (that is the structure theorem for sets of finite perimeter) and Lemma~\ref{mainconvII},  we conclude that for $\H^{n}$-a.e. $x\in\partial^*E_\infty$ there exists $B_r(x)\subset M$ so that $\partial^*E_\infty\cap B_r(x)=\partial E_\infty\cap B_r(x)$ is $C^1$ and furthermore $\Sigma_\infty=\partial E_\infty$ as varifolds in $B_r(x)$ (where in this last equality $\partial E_\infty$ is seen as a unit density varifold, the support of which  is a $C^1$ manifold).

For the generalized mean curvature of $\Sigma_\infty$ in $B_r(x)$ we then have that $\vec H_\infty=H_\infty\nu_\infty$ and $\mu_t\res(H_{t_i} \nu)\to \mu_\infty\res H_\infty\nu_\infty$. 
  Using this, the measure convergence (Lemma~\ref{mainconvII}) and the convergence $\mu_{t_i}\res \nu\to(\H^n\res\partial^* E_\infty)\res\nu_\infty$, we can argue as in \eqref{H+Plimit}, using again the Reshetnyak continuity and the lower semicontinuity, to conclude that
\[\int_{\partial^* E_\infty\cap B_r(x)}| H_\infty+ P(\nu_\infty)|d\H^n\le \liminf_i\int_{B_r(x)}|H_{t_i}+P(\nu)|d\mu_{t_i}.\]
Finally, using \eqref{to0} we obtain
\[\int_{\partial^* E_\infty\cap B_r(x)}| H_\infty+ P(\nu_\infty)|d\H^n=0,\]
which implies that $H_\infty(y)+P(\nu_\infty(y))=0$ for $\H^n$-a.e. $y\in \partial^* E_\infty\cap B_r(x)$.

Recalling the definition of a weak solution (Definition \ref{weak}) and that for the domain $\Omega_1$ (\eqref{uetou}, see also Lemma \ref{bdrystar}), we have therefore shown the following.

\begin{Theorem}\label{MOTSconv}
Let  $u\in C^{0,1}(\Omega_1)$ be a weak solution of $(**)$ (as  in Definition~\ref{weak}). Then $\partial^*(\Omega\setminus\Omega_1)$ is a generalized MOTS, as in Definition~\ref{weakMOTS}.
\end{Theorem}

\subsection{Remarks on further directions} 

Having established the proof of the main theorem, Theorem~\ref{main}, in this subsection we discuss in more detail some further directions as briefly mentioned at the end of the introduction.

As seen in Section \ref{properties}, the level sets $\partial^*E_t$, where $E_t=\{u>t\}$, of a weak solution $u$ of $(*)$ converge as finite perimeter sets to a generalized MOTS $\partial^*E_\infty$. In proving this, we have also showed  that the level sets $\Sigma_t=\{u=t\}$ converge also in the sense of varifolds, with their limit being the integral varifold $\Sigma_\infty=\var(\partial^* E_\infty)+ 2V$, where $V$ is some integral varifold in $\Omega$ and $\var(\partial^* E_\infty)$ is the varifold determined by $\partial^*E_\infty$. Note that, even though we know that $\Sigma_\infty$ has a generalized mean curvature, we can only make sense of the quantity $H+P$ in the `$\var(\partial^* E_\infty)$' part, as a notion of an outward pointing unit normal is required. We believe that $V=0$ and $\partial^*E_\infty=\partial E_\infty$ is actually a MOTS in the classical sense and therefore, as it lies outside the outermost MOTS, it is indeed the outermost MOTS. Such a result would be concluded if we  had some control over the singular set of null mean curvature flow, similar to that in  \cite{W00} for mean curvature flow. We explain the relation between the size of the singular set and the convergence to a MOTS below.

In Lemma \ref{eminu} we have showed that the level sets $\partial^*E_t$ satisfy a one-sided minimizing property, namely that of minimizing area plus {\it bulk energy P}. This property is inherited  from the level sets $\{U_\e=t\}$  (Lemma \ref{emin}). The level sets  $\{U_\e=t\}$ not only minimize ``area $+\int P$'' on the outside, but they also minimize (not only one-sided)  ``area $+\int \left(P-\frac{1}{|\ov\nabla U_\e|}\right)$'' (Remark \ref{minimrmk}).  The latter minimizing property would pass to the limit  if we have that $|\ov\nabla U_\e|^{-1}d\H^{n+1}\to|\ov\nabla U|^{-1}d\H^{n+1}$ as radon measures. This is indeed true in the case of mean curvature flow in $\R^n$ as proved in \cite{MS} and we believe that it also true in our case.  Now, if this convergence is true, and thus the level sets $\{u=t\}$ minimize ``area $+\int \left(P-\frac{1}{|\nabla u|}\right)$'', we can use the $L^1$-finiteness of $|\nabla u|^{-1}$ (Lemma \ref{normalintest}) to conclude that as $t\to \infty$ the limit minimizes ``area$+\int P$'' and therefore is a MOTS, with the regularity of the limit coming from the fact that it is a $C$-minimizing current as defined in Section \ref{eexistencesection}.  
Therefore, the question is how can we show the above convergence. It is not hard to check that the arguments from  \cite{MS} apply in our case, provided that the regularity theory of White \cite{W00} for the mean curvature flow is also true in our case. In particular we would like to have the following: There exists a singular set $S\subset \graph u$ of parabolic Hausdorff  dimension at most $n-1$ outside of which the sets $\{u=t\}$ are a smooth level set flow.

Finally, we would like to remark that if the level $\{u=t\}$ minimize ``area $+\int \left(P-\frac{1}{|\nabla u|}\right)$'' then we can define a weak solution of $(*)$ using this minimization property, as was done in \cite{MS} for the mean curvature flow (see also \cite{HI01, KM13} for the inverse mean curvature flow and the inverse null mean curvature flow). In \cite{MS}, this definition was used to show that the level set flow is unique and it is not hard to check that the methods from \cite{MS} can be applied to our case to show uniqueness.

\section*{Acknowledgments}
 We are  indebted to Felix Schulze for detailed conversations on his papers \cite{S08, MS} (the second co-authored with Jan Metzger), which were crucial for Sections \ref{eexistencesection} and \ref{properties} of our work and also to Jan Metzger who pointed to us the barrier constructions in his paper \cite{AM09} (co-authored with Lars Andersson) that inspired our barrier construction in Section \ref{outermostMOTSsection}. We would also like to thank Klaus Ecker, Gerhard Huisken,  Mat Langford, Ulrich Menne, Oliver Schn\"urer, Alexander Volkmann and Brian White for very useful and inspiring discussions on mean curvature flow, measure theory and parabolic PDEs.
\vspace{0in}
\bibliography{biblio}
\bibliographystyle{plain}

\end{document}